\documentclass[uplatex]{article}

\usepackage{amsthm,amsmath,amssymb}
\usepackage[all]{xy}
\usepackage{array, booktabs}
\usepackage{float}
\usepackage{mathtools}
\usepackage{extarrows}
\usepackage{tabularx}

\newcommand{\Z}{\mathbb{Z}}
\newcommand{\Q}{\mathbb{Q}}
\newcommand{\R}{\mathbb{R}}
\newcommand{\C}{\mathbb{C}}

\renewcommand{\H}{\mathbb{H}}
\renewcommand{\P}{\mathbb{P}}
\renewcommand{\i}{\sqrt{-1}}

\newcommand{\Hom}{\mathop{\mathrm{Hom}}\nolimits}
\newcommand{\End}{\mathop{\mathrm{End}}\nolimits}
\newcommand{\Aut}{\mathop{\mathrm{Aut}}\nolimits}

\newcommand{\Gal}{\mathop{\mathrm{Gal}}\nolimits}
\newcommand{\disc}{\mathop{\mathrm{disc}}\nolimits}

\newcommand{\trace}{\mathop{\mathrm{Tr}}\nolimits}
\newcommand{\norm}{\mathop{\mathrm{N}}\nolimits}
\newcommand{\relmiddle}[1]{\mathrel{}\middle#1\mathrel{}}

\def\i<#1>{\langle #1 \rangle}
\def\l<#1>{\left\langle #1 \right\rangle}
\makeatletter
\renewcommand{\theenumi}{{\upshape (\@roman\c@enumi)}}

\renewcommand{\p@enumii}{}
\renewcommand{\theenumii}{{\upshape (\@alph\c@enumii)}}

\makeatother

\theoremstyle{plain}
\newtheorem{theorem}{Theorem}[section]

\newtheorem{proposition}[theorem]{Proposition}
\newtheorem{lemma}[theorem]{Lemma}

\newtheorem{corollary}[theorem]{Corollary}
\newtheorem*{notation*}{Notation}

\begin{document}

\title{Intersection numbers of modular correspondences for genus zero modular curves}

\author{Yuya Murakami \thanks{Mathematical Inst. Tohoku Univ., 6-3, Aoba, Aramaki, Aoba-Ku, Sendai 980-8578, JAPAN
		\textit{E-mail address}: \texttt{yuya.murakami.s8@dc.tohoku.ac.jp}} }

\maketitle

\nocite{*}

\begin{abstract}
	In this paper, we introduce modular polynomials for the congruence subgroup $\Gamma_0(M)$ when $ X_0(M) $ has genus zero and therefore the polynomials are defined by a Hauptmodul of $ X_0(M) $.
	We show that the intersection number of two curves defined by two modular polynomials can be expressed as the sum of the numbers of $\mathrm{SL}_2(\Z)$-equivalence classes of positive definite binary quadratic forms over $\Z$.
	We also show that the intersection numbers can be also combinatorially written by Fourier coefficients of the Siegel Eisenstein series of degree
	2, weight 2 with respect to $\mathrm{Sp}_2(\Z)$.
\end{abstract}

\section{Introduction}\label{sec:intro}

For $N \in \Z_{>0}$, we consider the following function:
\[
\Phi_N(j(E), j(E^\prime)) \coloneqq \prod_{[f, E_1^\prime]} (j(E) - j(E_1^\prime))
\]
where $E, E^\prime$ are elliptic curves over $ \C $ and $[f, E_1^\prime]$ is the equivalence class of an ordered pair $(f, E_1^\prime)$ of degree $N$ isogeny $f \colon E_1^\prime \to E^\prime$.
It is well-known that the function $\Phi_N(X, Y)$ is a symmetric polynomial in $\Z[X, Y]$. We call this polynomial the modular polynomial of degree $ N $.
The symbol $ T_N $ denotes the affine plane algebraic curve defined by the modular polynomial $\Phi_N(X, Y)$.

In \cite{Hur} (see also Proposition 2.4 of \cite{GK}), Hurwitz computed the intersection numbers of two affine algebraic curves $ T_{N_1}$ and $ T_{N_2}$, which is defined by 
$(T_{N_1}, T_{N_2}) \coloneqq \mathrm{dim}_{\C}\C[X, Y]/(\Phi_{N_1}, \Phi_{N_2})$
in this case. 

\begin{theorem}[Hurwitz]\label{prop:Hurwitz}
	The curves $ T_{N_1}$ and $ T_{N_2} $ intersect properly if and only if $N_1 N_2$ is not a square. In this case, their intersection number is
	\[
	\begin{split}
	(T_{N_1}\cdot T_{N_2})
	&= \sum_{t^2 < 4N_1 N_2} \sum_{d \mid (N_1, N_2, t)} d\cdot H\left( \dfrac{4N_1 N_2-t^2}{d^2} \right),
	\end{split}
	\]
	where $H(D)$ is the number of $\mathrm{SL}_2(\Z)$-equivalence classes of positive definite binary quadratic forms over $\Z$ with determinant $D$, counting the forms equivalent to $ex_1^2 + ex_2^2$ and $ex_1^2 + ex_1 x_2 + ex_2^2$ with multiplicity $1/2$ and $1/3$ respectively. 
\end{theorem}

A detailed proof will be found in \cite{GK} or \cite{Vog} with the result in \cite{Gor}.
Another proof will be found in \cite{Ling} by calculating intersection multiplicities at cusps and intersection number of $ T_{N_1}$ and $ T_{N_2}$ as cycles on $ \P^1 \times \P^1 $.
In \cite{GK}, they also computed the arithmetic intersection numbers of three cycles $ T_{N_1}, T_{N_2}$ and $ T_{N_3}$ on the 3-dimensional scheme $ \mathrm{Spec}\hspace{0.5mm} \Z[X, Y] $.

The curve $ T_N $ can be considered as an algebraic cycle in the ambient space $ Y(1) \times Y(1) = \C \times \C $ and the intersection of two cycles is taking in $ Y(1) \times Y(1)=\C \times \C $.
It seems interesting to study the intersection of a similar cycles in more general ambient space.

In this paper, we replace $Y(1)$ with any $Y_0(M)$ whose compactification has genus zero and study an analogue of $T_N$ defined in 
$Y_0(M)\times Y_0(M)$.
In this situation, we obtain a Hauptmodul $t$ of $Y_0(M)$, which generates the field of meromorphic function on a canonical compactification of $Y_0(M)$. 
Let us consider the following function for $Y_0(M)$ defined by 
\[
\Phi_N^{\Gamma_0(M)}(t(E, C), t(E^\prime, C^\prime))\coloneqq
\displaystyle \prod_{[f, E_1^\prime, C_1^\prime]} (t(E, C) - t(E_1^\prime, C_1^\prime))
\]
where $(E, C)$ is an elliptic curve with level structure for $\Gamma_0(M)$, that is, $ C $ is a cyclic subgroup of $ E $ of order $ M $, and $[f, E_1^\prime, C_1^\prime]$ are the equivalence class of an ordered pair $(f, E_1^\prime)$ of degree $N$ isogeny $f \colon E_1^\prime \to E^\prime$ such that $f(C_1^\prime)=C^\prime$.
Two objects $(f_1, E_1, C_1) \cong (f_2, E_2, C_2)$ are equivalent if and only if there exists an isogeny $g \colon E_1 \to E_2$ such that $f_1 = f_2 \circ g$ and $g(C_1)=C_2$.

In the definition of $\Phi_N^{\Gamma_0(M)}$, we choose $ t $ so that it has a simple pole at $\infty$ with residue 1 and is holomorphic on $X_0(M) \setminus \{\infty\}$ and whose Fourier expansion with respect to $q \coloneqq e^{2\pi\sqrt{-1}\tau}$ has integral coefficients.
We give this $ t $ explicitly in section ~\ref{sec:modpoly}.
It will turn out that the function $\Phi_N^{\Gamma_0(M)}$ is a symmetric polynomial in $ \Z[X, Y] $ and we call it the modular polynomial of level $ N $ for $ \Gamma_0(M) $.
The symbol $ T_N^{\Gamma_0(M)} $ denotes the affine plane algebraic curve defined by the modular polynomial $\Phi_N^{\Gamma_0(M)}(X, Y)$.
The curve $ T_N^{\Gamma_0(M)} $ can be considered as an algebraic cycle in the ambient space  $ Y_0(M) \times Y_0(M) \subset \C \times \C $ under $ t \times t $.

In the case when two affine algebraic curves $T_{N_1}^{\Gamma_0(M)}$ and $T_{N_2}^{\Gamma_0(M)}$ in $ t(Y_0(M)) \times t(Y_0(M))$ intersect properly, the intersection multiplicity of $T_{N_1}^{\Gamma_0(M)}$ and $T_{N_2}^{\Gamma_0(M)}$ at $ (x_0, y_0) \in t(Y_0(M)) \times t(Y_0(M))$ is defined by 
\[
(T_{N_1}^{\Gamma_0(M)} \cdot T_{N_2}^{\Gamma_0(M)})_{(x_0, y_0)} \coloneqq \mathrm{dim}_{\C}\C[[X-x_0, Y-y_0]] / (\Phi_{N_1}^{\Gamma_0(M)}, \Phi_{N_2}^{\Gamma_0(M)}).
\]
The intersection number of $T_{N_1}^{\Gamma_0(M)}$ and $T_{N_2}^{\Gamma_0(M)}$ on $ t(Y_0(M)) \times t(Y_0(M))$ is defined by
\[
(T_{N_1}^{\Gamma_0(M)} \cdot T_{N_2}^{\Gamma_0(M)}) \coloneqq \sum_{(x_0, y_0) \in Y_0(M) \times Y_0(M)} (T_{N_1}^{\Gamma_0(M)}, T_{N_2}^{\Gamma_0(M)})_{(x_0, y_0)}
\]
except for $ M=1 $.

The  main result in this paper is the following theorem.

\begin{theorem}\label{thm:main}
For two positive integers $ N_1, N_2 $ which are relatively prime to $ M $, the curves $T_{N_1}^{\Gamma_0(M)}$ and $T_{N_2}^{\Gamma_0(M)}$ intersect properly if and only if $N_1 N_2$ is not a square. In this case, their intersection number on $ t(Y_0(M)) \times t(Y_0(M))$ is
\[
\begin{split}
&(T_{N_1}^{\Gamma_0(M)} \cdot T_{N_2}^{\Gamma_0(M)}) \\
&=\sum_{x \in \Z,\ x^2 < 4N_1 N_2} \sum_{Z \mid (N_1, N_2, x)}
Z \cdot H^M \left( \frac{4N_1 N_2- x^2}{Z^2} \right) 
\end{split}
\]
where $H^M(D)$ is the following sum with respect to $\Gamma_0(M)$-equivalence classes of positive definite binary quadratic forms $ Q $ such that $Q= [Ma, b, c] $ for some integers $ a, b, c $ and whose determinant is $ D $ :
\[
H^M(D) \coloneqq \sum_{Q= [Ma, b, c],\ \det Q =D}
\frac{1}{[\Gamma_0(M)_Q : \{\pm1 \}]},
\]
and $ \Gamma_0(M)_Q $ is the stabilizer subgroup of $ \Gamma_0(M) $ with respect to $ Q $.

For the case when $ M=p $ is a prime, then this intersection number is equal to
\[
\sum_{x^2 < 4N_1 N_2} A^p(4N_1 N_2- x^2) \sum_{d \mid (N_1, N_2, x)}
d \cdot H \left( \frac{4N_1 N_2- x^2}{d^2} \right)
\]
where
\[
A^p(D) \coloneqq
\begin{cases}
1 + \chi_{D}(p) & \text{ if } v \coloneqq \lfloor \frac{\mathrm{ord}_p D}{2} \rfloor =0 \\
1+ \frac{p}{1+ \frac{1}{p^v} \frac{1-p^v}{1-p} \left( 1-\frac{\chi_{D}(p)}{p} \right) } & \text{ if } v \ge 1
\end{cases}
\]
and $ \chi_{D} $ is the quadratic character associated to the imaginary quadratic field $ \Q(\sqrt{-D}) $.
\end{theorem}

The computation will be carried out with a similar manner as in \cite{GK}, \cite{Vog} and \cite{Gor}.
However the situation becomes more complicated than their case because of the level structure and thereby we need a careful analysis.   
To handle it we will bypass to several modular curves naturally arose from the level structures which will be discussed in Section 2. 
Then the various sets and invariants will show up in Section 2, 3. Combining all with some class numbers defined by a motivation coming from [HZ76]  we will reach to the desired formula.

Next we explain a relation between the intersection numbers and Fourier coefficients of Siegel Eienstein series which would be related to 
Kudla's problem \cite{Kudla} for the orthogonal group $SO(3, 2) \sim Sp_2(\Z)$ though he studied the intersections on compact Shimura varieties. Henceforth we follow the notation in \cite{Nagaoka}.

Let $\mathbb{H}_2=\{Z=X+Y\sqrt{-1}\in M_2(\C)\ |\ {}^tZ=Z,\ Y>0\}$ be the Siegel upper half plane of degree 2.
Let $E^{(2)}_2(Z,s)$ be the Siegel Eisenstein series of degree 2, weight 2 with respect to $\mathrm{Sp}_2(\Z)\subset M_4(\Z)$ where $s$ is the complex parameter. 
It is known that $E^{(2)}_2(Z,s)$ is absolutely convergent if ${\rm Re}(s)>\frac{1}{2}$ and further it is extended to the whole space of 
complex numbers in $s$. As mentioned in \cite{Nagaoka} Kohnen  showed that $E^{(2)}_2(Z,s)$ is finite  at $s=0$ that means 
$E^{(2)}_2(Z,0)$ is a smooth function on $\mathbb{H}_2$. Let us consider the Fourier expansion 
$$E^{(2)}_2(Z,0)=\sum_{T\in {\rm Sym}_2(\Z)}C(T,Y)e^{2\pi \sqrt{-1}{\rm tr}(TZ)},\ Z=X+Y\sqrt{-1}$$
where ${\rm Sym}_2(\Z)$ stands for the set of all symmetric half-integral 2 by 2 matrices over $\Z$ so that 
the diagonal entries are defined over $\Z$ while the anti-diagonal entries are defined over $\frac{1}{2}\Z$.  
We denote by ${\rm Sym}_2(\Z)_{>0}$ the subset of ${\rm Sym}_2(\Z)$ consisting of all positive definite matrices.  
For any $T\in {\rm Sym}_2(\Z)_{>0}$ it is shown that $T$-th Fourier coefficient $C(T,Y)$ is independent of $Y$ and therefore 
we may set $C(T):=C(T,Y)$. We also denote by $\chi_T$ the quadratic character associated to the imaginary 
quadratic field $\Q(\sqrt{-\det(2T)})$ for 
$T\in {\rm Sym}_2(\Z)_{>0}$.  
Then our next main result is as follows. 

\begin{theorem}\label{thm:main2}
Keep the notation in Theorem \ref{thm:main}. For the case when $ M=p $ is a prime, it holds that 
$$(T^{\Gamma_0(M)}_{N_1}\cdot T^{\Gamma_0(M)}_{N_2})=
\frac{1}{288}
\sum_{ T\in {\rm Sym}_2(\Z)_{>0},\ {\rm diag}(T)=(N_1/2, N_2/2)}
A^p(\det(2T)) C(T)$$
where ${\rm diag}(T)$ stands for the diagonal part of $T$. 	
\end{theorem}

This theorem explains our intersection numbers can be combinatorially written by Fourier coefficients of Eisenstein series of 
``level one". However it seems natural to expect a direct relation of them with Fourier coefficients of a single Siegel modular form with a 
non-trivial level.   
In fact, according to Kudla's problem, if we can extend the results in \cite{Kudla} in case of $SO(3,2)$ to the non-compact quotient case, 
it is plausible to believe the existence of the Siegel Eisenstein series of degree 2, weight 2 with a non-trivial level 
whose Fourier coefficients are directly related to the intersection number in the above theorem. This will be discussed somewhere else.  

We will organize this paper as follows:
In Section ~\ref{sec:modpoly}, we study the basic properties of modular polynomials for modular curves of genus zero including other curves as $ X_1(M), X(M) $.
In Section ~\ref{sec:quad_form}, we study $\Gamma_0(M)$-equivalence classes of primitive positive definite binary quadratic forms over $ \Z $.
In Section ~\ref{sec:intnum}, we calculate the intersection multiplicity of $T_{N_1}^{\Gamma_0(M)}$ and $T_{N_2}^{\Gamma_0(M)}$ when $ N_1 N_2 $ is a non-square and $(M, N_1 N_2)=1$ and prove Theorem \ref{thm:main} and Theorem \ref{thm:main2}.

\section{Modular polynomials}\label{sec:modpoly}

In this Section, we introduce modular polynomials for some congruence subgroups in $ \mathrm{SL}_2(\Z) $ and study basic properties of them. 
First of all, we define other modular polynomials.

Let $ \mathrm{M}_2(\Z) $ be the set of square matrices of size 2 and $ \mathrm{SL}_2(\Z) $ be the special linear group.
We set congruence subgroups in $\mathrm{SL}_2(\Z)$ by
\[
\Gamma_0(M) \coloneqq \left\{ A \in \mathrm{SL}_2(\Z) \relmiddle| A \equiv \begin{pmatrix} * & * \\ 0 & * \end{pmatrix} \bmod M \right\},
\]
\[
\Gamma_1(M) \coloneqq \left\{ A \in \mathrm{SL}_2(\Z) \relmiddle| A \equiv \begin{pmatrix} 1 & * \\ 0 & 1 \end{pmatrix} \bmod M \right\},
\]
\[
\Gamma(M) \coloneqq \left\{ A \in \mathrm{SL}_2(\Z) \relmiddle| A \equiv \begin{pmatrix} 1 & 0 \\ 0 & 1 \end{pmatrix} \bmod M \right\}.
\]
These congruence subgroups yield the following modular curves obtained by quoting the complex upper half plane $ \mathbb{H}=\{z=x+\sqrt{-1}y\in \C |\ y>0 \} $ by them as: 
\[
Y_0(M) \coloneqq \Gamma_0(M)\backslash\H, \quad
Y_1(M) \coloneqq \Gamma_1(M)\backslash\H, \quad
Y(M) \coloneqq \Gamma(M)\backslash\H.
\]
Let $Y = Y_0(M), Y_1(M)$ or $Y(M)$ whose compactification has genus zero. 
The list of such $ M $ will be found in the last paragraph of Section 3 in \cite{Seb}.
In the case when $Y = Y_0(M)$, $ M $ satisfies $ 1\le M \le 10$ or $ M= 12, 13, 16, 18$ or 25.
In the case when $Y = Y_1(M)$, $ M $ satisfies $ 1\le M \le 10$ or $ M= 12$ and in the case when $Y = Y(M)$, $ M $ satisfies $ 1\le M \le 5 $.
In this situation, we obtain a Hauptmodul $t$ of $Y$, which generates the function field $\C(Y)$. 
Let $ N $ be a positive integer.
In the case when $Y = Y_0(M)$, we define the function
\[
\Phi_N^{\Gamma_0(M)}(t(E, C), t(E^\prime, C^\prime)) =
\displaystyle \prod_{[f, E_1^\prime, C_1^\prime] \in I_N^{\Gamma_0(M)}(E, C)_{\mathrm{isog}}} (t(E, C) - t(E_1^\prime, C_1^\prime))
\]
where  $(E, C), (E^\prime, C^\prime)$ are elliptic curves with level structure for $\Gamma_0(M)$,
\[
I_N^{\Gamma_0(M)}(E, C)_{\mathrm{isog}} \coloneqq \{ (f, E_1, C_1) \mid f \colon E_1\to E \text{ isogeny}, f(C_1)=C \}/\cong
\]
and $(f_1, E_1, C_1) \cong (f_2, E_2, C_2)$ if and only if there exists an isogeny $g \colon E_1 \to E_2$ such that $f_1 = f_2 \circ g$ and $g(C_1)=C_2$.

Similarly, in the case when $Y = Y_1(M)$, we define
\[
\Phi_N^{\Gamma_1(M)}(t(E, Q), t(E^\prime, Q^\prime)) \coloneqq
\displaystyle \prod_{[f, E_1^\prime, Q_1^\prime] \in I_N^{\Gamma_1(M)}(E, Q)_{\mathrm{isog}}} (t(E, Q) - t(E_1^\prime, Q_1^\prime))
\]
where  $(E, Q), (E^\prime, Q^\prime)$ are elliptic curves with level structure for $\Gamma_1(M)$, that is, $ Q $ and $ Q^\prime $ are points of order $ M $ in $ E $ and $ E^\prime $ respectively.
We also define
\[
I_N^{\Gamma_1(M)}(E, Q)_{\mathrm{isog}} \coloneqq \{ (f, E_1, Q_1) \mid f \colon E_1\to E \text{ isogeny}, f(Q_1)=Q \}/\cong
\]
where $(f_1, E_1, Q_1) \cong (f_2, E_2, Q_2)$ if and only if there exists an isogeny $g \colon E_1 \to E_2$ such that $f_1 = f_2 \circ g$ and $g(Q_1)=Q_2$.

In the case when $Y = Y(M)$, we define
\[
\Phi_N^{\Gamma(M)}(t(E, P, Q), t(E^\prime, P^\prime, Q^\prime)) \coloneqq
\displaystyle \prod_{[f, E_1^\prime, P_1^\prime, Q_1^\prime] \in I_N^{\Gamma(M)}(E, P, Q)_{\mathrm{isog}}} (t(E, P, Q) - t(E_1^\prime, P_1^\prime, Q_1^\prime))
\]
where  $(E, P, Q), (E^\prime, P^\prime, Q^\prime)$ are elliptic curves with level structure for $\Gamma(M)$, that is, $ P, Q $ is points of order $ M $ in $ E $ such that $ e_M(P, Q)= e^{2\pi\sqrt{-1}/M} $ where $ e_M $ is the Weil pairing. The points $ P^\prime$ and $ Q^\prime $ satisfy a similar condition.
We also define
\[
I_N^{\Gamma(M)}(E, P, Q)_{\mathrm{isog}} \coloneqq \{ (f, E_1, P_1, Q_1) \mid f \colon E_1\to E \text{ isogeny}, f(P_1)=P, f(Q_1)=Q \}/\cong
\]
where $(f_1, E_1, P_1, Q_1) \cong (f_2, E_2, P_2, Q_2)$ if and only if there exists an isogeny $g \colon E_1 \to E_2$ such that $f_1 = f_2 \circ g, g(P_1)= P_1$ and $g(Q_1)=Q_2$.

For $ Y=Y(1)$, we define
\[
\Psi_N(j(E), j(E^\prime)) \coloneqq \prod_{[f, E_1^\prime]} (j(E) - j(E_1^\prime))
\]
where $E, E^\prime$ are elliptic curves and $[f, E_1^\prime]$ is the equivalence class of an ordered pair $(f, E_1^\prime)$ of degree $N$ cyclic isogeny $f \colon E_1^\prime \to E^\prime$.
It is a classical result that $ \Phi_N = \prod_{N_1^2 \mid N} \Psi_{N/N_1^2}(X, Y) $, $ \Psi_N(X, j) $ is irreducible over $ \C(j) $ and the vanishing set of $\Psi_N(X, Y)$ is the image of the composition of the following two maps
\[
\begin{array}{ccccc}
Y_0(N) & \longrightarrow & Y(1) \times Y(1) & \xrightarrow{\sim} & \C \times \C \\
{[E, C]} & \longmapsto & ([E], [E/C]) & \longmapsto & (j(E), j(E/C)) 
\end{array}
\]
where $[E, C]$ is the equivalence class of the elliptic curve with level structure for $\Gamma_0(N)$.
A proof will be found in \cite{Vog}.

To factorize $ \Phi_N^{\Gamma_0(M)}, \Phi_N^{\Gamma_1(M)}, \Phi_N^{\Gamma(M)} $ into irreducible components, we define
\[
P_N^{\Gamma_0(M)}(E, C)_{\mathrm{isog}} \coloneqq \{ [f, E_1, C_1] \in I_N^{\Gamma_0(M)}(E, C)_{\mathrm{isog}} \mid f \text{ cyclic isogeny} \},
\]
\[
P_N^{\Gamma_1(M)}(E, Q)_{\mathrm{isog}} \coloneqq \{ [f, E_1, Q_1] \in I_N^{\Gamma_1(M)}(E, Q)_{\mathrm{isog}} \mid f \text{ cyclic isogeny} \}
\]
and
\[
P_N^{\Gamma(M)}(E, P, Q)_{\mathrm{isog}} \coloneqq \{ [f, E_1, P_1, Q_1] \in I_N^{\Gamma(M)}(E, P, Q)_{\mathrm{isog}} \mid f \text{ cyclic isogeny} \}.
\]
Accordingly, we also define
\[
\Psi_N^{\Gamma_0(M)}(t(E, C), t(E^\prime, C^\prime)) =
\displaystyle \prod_{[f, E_1^\prime, C_1^\prime] \in P_N^{\Gamma_0(M)}(E, Q)_{\mathrm{isog}}} (t(E, C) - t(E_1^\prime, C_1^\prime)),
\]
\[
\Psi_N^{\Gamma_1(M)}(t(E, Q), t(E^\prime, Q^\prime)) =
\displaystyle \prod_{[f, E_1^\prime, Q_1^\prime] \in P_N^{\Gamma_1(M)}(E, C)_{\mathrm{isog}}} (t(E, Q) - t(E_1^\prime, Q_1^\prime))
\]
and
\[
\Psi_N^{M}(t(E, P, Q), t(E^\prime, P^\prime, Q^\prime)) \coloneqq
\displaystyle \prod_{[f, E_1^\prime, P_1^\prime, Q_1^\prime] \in P_N^{M}(E, P, Q)_{\mathrm{isog}}} (t(E, P, Q) - t(E_1^\prime, P_1^\prime, Q_1^\prime))
\]
where $(E, C), (E, Q), (E, P, Q)$ are elliptic curves with level structure for $\Gamma_0(M), \Gamma_1(M)$, $\Gamma(M)$ respectively.

We define
\[
Y_0^0(M, N) \coloneqq \Gamma_0^0(M, N)\backslash\H, \quad
Y_1^0(M, N) \coloneqq \Gamma_1^0(M, N)\backslash\H, \quad
Y^0(M, N) \coloneqq \Gamma^0(M, N)\backslash\H,
\]
\[
\Gamma_0^0(M, N) \coloneqq \Gamma_0(M) \cap \Gamma^0(N), \quad
\Gamma_1^0(M, N) \coloneqq \Gamma_1(M) \cap \Gamma^0(N), \quad
\Gamma^0(M, N) \coloneqq \Gamma(M) \cap \Gamma^0(N),
\]
\[
\Gamma^0(N) \coloneqq \left\{ A \in \mathrm{SL}_2(\Z) \relmiddle| A \equiv \begin{pmatrix} * & 0 \\ * & * \end{pmatrix} \bmod N \right\},
\]
\[
\Gamma^1(M) \coloneqq \left\{ A \in \mathrm{SL}_2(\Z) \relmiddle| A \equiv \begin{pmatrix} 1 & 0 \\ * & 1 \end{pmatrix} \bmod N \right\}.
\]
Elliptic curves with level structure for $\Gamma_0^0(M, N), \Gamma_1^0(M, N), \Gamma^0(M, N)$ are ordered sets $(E, C, D), (E, C, Q), (E, C, P, Q)$ where $C$ is a cyclic subgroup of $E$ of order $N$, $D$ is a cyclic subgroup of $E$ of order $M$ such that $C \cap D = \{0\}$, $Q$ is a point of order $M$ such that $C \cap \langle Q \rangle = \{0\}$, $P$ is a point of order $M$ with Weil pairing $e_M(P, Q) = e^{2\pi\sqrt{-1}/M}$.
Two elliptic curves $(E, C, Q), (E^\prime, C^\prime, Q^\prime)$ with level structure for $\Gamma_1^0(M, N)$ are equivalent if some isomorphism $E \cong E^\prime$ of elliptic curves induces group isomorphism $C \cong C^\prime$ and maps $Q$ to $Q^\prime$.
We also define isomorphisms of elliptic curves with level structure for $\Gamma_0^0(M, N), \Gamma^0(M, N)$ by same way.
We prove that $Y_0^0(M, N), Y_1^0(M, N), Y^0(M, N)$ are the moduli spaces of elliptic curves with level structure for $\Gamma_0^0(M, N), \Gamma_1^0(M, N), \Gamma^0(M, N)$ respectively.

Let us start with the following Lemma which would be almost obvious.

\begin{lemma}\label{lem:moduli_sp}
	The following maps are bijective:
	\[
	\begin{array}{ccc}
	\Gamma_0^0(M, N)\backslash\H & \longrightarrow &
	\{ \text{elliptic curves with level structure for }\Gamma_0^0(M, N) \}/\cong \\
	\Gamma_0^0(M, N)\tau & \longmapsto & {[E_{\tau}, \langle \frac{\tau}{N} + \Lambda_{\tau} \rangle, \langle \frac{1}{M} + \Lambda_{\tau} \rangle]},
	\end{array}
	\]
	\[\textsl{}
	\begin{array}{ccc}
	\Gamma_1^0(M, N)\backslash\H & \longrightarrow &
	\{ \text{elliptic curves with level structure for }\Gamma_1^0(M, N) \}/\cong \\
	\Gamma_1^0(M, N)\tau & \longmapsto & {[E_{\tau}, \langle \frac{\tau}{N} + \Lambda_{\tau} \rangle, \frac{1}{M} + \Lambda_{\tau}]},
	\end{array}
	\]
	\[
	\begin{array}{ccc}
	\Gamma^0(M, N)\backslash\H & \longrightarrow &
	\{ \text{elliptic curves with level structure for }\Gamma^0(M, N) \}/\cong \\
	\Gamma^0(M, N)\tau & \longmapsto & {[E_{\tau}, \langle \frac{\tau}{N} + \Lambda_{\tau} \rangle, \frac{\tau}{M} + \Lambda_{\tau}, \frac{1}{M} + \Lambda_{\tau}]}
	\end{array}
	\]
	where $E_{\tau} \coloneqq \C/\Lambda_{\tau}, \Lambda_{\tau} \coloneqq \Z + \Z\tau$.
\end{lemma}

By Lemma \ref{lem:moduli_sp}, the vanishing sets of $\Psi_N^{\Gamma_0(M)}(X, Y), \Psi_N^{\Gamma_1(M)}(X, Y)$, $\Psi_N^{M}(X, Y)$ in $ t(Y_0(M))^2, t(Y_1(M))^2, t(Y(M))^2 $ coincide with the images of the following maps
\[
\begin{array}{ccccc}
Y_0^0(M, N) & \longrightarrow & Y_0(M) \times Y_0(M) & \longrightarrow & \C \times \C \\
{[E, C, D]} & \longmapsto & ([E, D], [E/C, (D+C)/C]) & \longmapsto & (t(E, D), t(E/C, (D+C)/C)),
\end{array}
\]
\[
\begin{array}{ccccc}
Y_1^0(M, N) & \longrightarrow & Y_1(M) \times Y_1(M) & \longrightarrow & \C \times \C \\
{[E, C, Q]} & \longmapsto & ([E, Q], [E/C, (Q+C)/C]) & \longmapsto & (t(E, Q), t(E/C, (Q+C)/C)),
\end{array}
\]
\[
\begin{array}{ccc}
Y^0(M, N) & \longrightarrow & Y(M) \times Y(M) \\
{[E, C, P, Q]} & \longmapsto & ([E, P, Q], [E/C, (P+C)/C, (Q+C)/C]) \\
& \longrightarrow & \C \times \C \\
& \longmapsto & (t(E, P, Q), t(E/C, (P+C)/C, (Q+C)/C))
\end{array}
\]
respectively.

We give two lemmata in order to calculate the group indeces $[\Gamma_0(M):\Gamma_0^0(M, N)], [\Gamma_1(M):\Gamma_1^0(M, N)]$ and $[\Gamma(M):\Gamma^0(M, N)]$.

\begin{lemma}\label{lem:elem_int}
	Let $a, b \in \Z$.
	If $(M, N)=1$ and $(a, b, N)=1$, then there exist $m, n \in \Z$ such that
	\[
	(m, n)=1, \quad
	m \equiv \begin{cases}
	0 \bmod M \\
	a \bmod N,
	\end{cases}
	n \equiv \begin{cases}
	1 \bmod M \\
	b \bmod N.
	\end{cases}
	\]
\end{lemma}

\begin{proof}
	Since $(M, N)=1$, there exist $a^\prime, b^\prime \in \Z$ such that
	\[
	a^\prime \neq 0, \quad
	a^\prime \equiv \begin{cases}
	0 \bmod M \\
	a \bmod N,
	\end{cases}
	b^\prime \equiv \begin{cases}
	1 \bmod M \\
	b \bmod N.
	\end{cases}
	\]
	Then $(a^\prime, b^\prime, N)=(a, b, N)=1$.
	By replacing $a, b$ to $a^\prime, b^\prime$, we can assume that 
	$a \neq 0, a \equiv 0 \bmod M, b \equiv 1 \bmod M$.
	
	Let $g \coloneqq (a, b)$. 
	By Chinese remainder theorem, there exists $t \in \Z$ such that
	\[
	(m, n)=1, \quad
	m \equiv \begin{cases}
	1 \bmod p & \text{for primes } p \nmid M, p \mid g \\
	0 \bmod M \\
	0 \bmod p & \text{for primes } p \nmid M, p \nmid g, p \mid a.
	\end{cases}
	\]
	Let $m \coloneqq a, n \coloneqq b+tN$. Then $n \equiv 1 \bmod M$.
	
	Suppose that $p \mid a$.
	If $p \nmid M, p \mid g$, then $p \mid b, p \nmid t$.
	Since $(g, N)=1$, $p \nmid N$.
	Thus $p \nmid n$.
	If $p \mid M$, then $p \mid t$.
	Since $(b, M)=1$, $p \nmid b$.
	Thus $p \nmid n$.
	If $p \nmid M, p \nmid g$, then $p \mid t, p \nmid b$.
	Thus $p \nmid n$.
	Therefore $(m, n)=1$.
\end{proof}

\begin{lemma}\label{lem:Gamma_prod}
	If $(M, N)=1$, then $\Gamma_0(N)\Gamma(M) = \mathrm{SL}_2(\Z)$.
	Especially, $\Gamma_0(N)\Gamma_0(M) =\Gamma_0(N)\Gamma_1(M) = \mathrm{SL}_2(\Z)$.
\end{lemma}

\begin{proof}
	Let $\begin{pmatrix} a & b \\ c & d \end{pmatrix}
	\in \mathrm{SL}_2(\Z)$.
	By Lemma \ref{lem:elem_int}, there exist $l, n \in \Z$ such that
	\[
	(l, n)=1, \quad
	l \equiv \begin{cases}
	0 \bmod M \\
	b \bmod N,
	\end{cases}
	n \equiv \begin{cases}
	1 \bmod M \\
	-a \bmod N.
	\end{cases}
	\]
	Since $(n, lM)=1$, there exist $k, m \in \Z$ such that $kn - mlM = 1$.
	Then
	\[
	\begin{pmatrix} k & l \\ mM & n \end{pmatrix} \in \Gamma(M), \quad
	\begin{pmatrix} a & b \\ c & d \end{pmatrix}
	\begin{pmatrix} k & l \\ mM & n \end{pmatrix}
	= \begin{pmatrix} * & al+bn \\ * & * \end{pmatrix} \in \Gamma^0(N).
	\]
\end{proof}

\begin{proposition}
	For $M, N \in \Z$, it holds that
	\[
	\begin{split}
	& [\Gamma_0(M):\Gamma_0^0(M, N)]
	=[\Gamma_1(M):\Gamma_1^0(M, N)] \\
	&=[\Gamma(M):\Gamma^0(M, N)]
	=N \prod_{p \mid N, p \nmid M} (1+p^{-1}).
	\end{split}
	\]
	where the product is taken over all prime divisors of $N$ not dividing $M$.
	Especially, if $(M, N)=1$, then this number is equal to $[\mathrm{SL}_2(\Z):\Gamma_0(N)]$.
\end{proposition}

\begin{proof}
	If $(M, N)=1$, then $\Gamma_0(N)\Gamma_0(M) = \mathrm{SL}_2(\Z)$ by Lemma \ref{lem:Gamma_prod}. Thus
	\[
	\begin{split}
	& [\Gamma_0(M):\Gamma_0^0(M, N)]
	=\#\Gamma^0(N)\backslash\Gamma^0(N)\Gamma_0(M) \\
	&=[\mathrm{SL}_2(\Z):\Gamma^0(N)]
	=N \prod_{p \mid N} (1+p^{-1}).
	\end{split}
	\]
	
	For arbitrary $M, N \in \Z_{>0}$, let $g \coloneqq (M, N)$.
	There is a bijection
	\[
	\begin{array}{ccc}
	\Gamma_0^0(M, N)\backslash\Gamma_0^0(M, N/g) & \longrightarrow &
	\displaystyle\frac{N}{g}\Z/N\Z \\
	\Gamma_0^0(M, N) \begin{pmatrix} a & b \\ c & d \end{pmatrix}
	& \longmapsto & a^{-1}b \bmod N.
	\end{array}
	\]
	Thus
	\[
	\begin{split}
	&[\Gamma_0(M):\Gamma_0^0(M, N)]
	=[\Gamma_0^0(M, N/g) : \Gamma_0^0(M, N)][\Gamma_0(M):\Gamma_0^0(M, N/g)] \\
	&= g \cdot \frac{N}{g} \prod_{p \mid N/g} (1+p^{-1}) 
	= N \prod_{p \mid N, p \nmid M} (1+p^{-1}).
	\end{split}
	\]
\end{proof}

\begin{lemma}\label{lem:ellmatcorres}
	\begin{enumerate}
		\item\label{item:lem:ellmatcorres1} Let $(E, C)$ be an elliptic curve with level structure for $\Gamma_0(M)$. Then, there is the canonical bijection between the following sets:
		\begin{enumerate}
			\item $I_N^{\Gamma_0(M)}(E, C)_{\mathrm{isog}}$,
			\item $I^{\Gamma_0(M)}_{N, \mathrm{mat}} \coloneqq \Gamma_0(M)\backslash\left\{ A \in \mathrm{M}_2({\Z}) \relmiddle|
			\begin{array}{l}
			\det A =N,\\
			d \in (\Z/M\Z)^\times,
			\end{array}
			A \equiv 
			\begin{pmatrix} * & * \\ 0 & d \end{pmatrix} \bmod M
			\right\}$.
		\end{enumerate}
		Same bijection induces the bijection between the following sets:
		\begin{enumerate}
			\item $P_N^{\Gamma_0(M)}(E, C)_{\mathrm{isog}}$,
			\item $P^{\Gamma_0(M)}_{N, \mathrm{mat}} \coloneqq \{ [A] \in I^{\Gamma_0(M)}_{N, \mathrm{mat}} \mid A \text{ is primitive} \}$.
		\end{enumerate}
		\item\label{item:lem:ellmatcorres2} Let $(E, Q)$ be an elliptic curve with level structure for $\Gamma_1(M)$. Then, there is the canonical bijection between the following sets:
		\begin{enumerate}
			\item $I_N^{\Gamma_1(M)}(E, Q)_{\mathrm{isog}}$,
			\item $I^{\Gamma_1(M)}_{N, \mathrm{mat}} \coloneqq \Gamma_1(M)\backslash\left\{ A \in  \mathrm{M}_2({\Z}) \relmiddle| \det A =N, A \equiv 
			\begin{pmatrix} * & * \\ 0 & 1 \end{pmatrix} \bmod M
			\right\}$.
		\end{enumerate}
		Same bijection induces the bijection between the following sets:
		\begin{enumerate}
			\item $P_N^{\Gamma_1(M)}(E, Q)_{\mathrm{isog}}$,
			\item $P^{\Gamma_1(M)}_{N, \mathrm{mat}} \coloneqq \{ [A] \in I^{\Gamma_1(M)}_{N, \mathrm{mat}} \mid A \text{ is primitive} \}$.
		\end{enumerate}
		\item\label{item:lem:ellmatcorres3} Let $(E, P, Q)$ be an elliptic curve with level structure for $\Gamma(M)$. Then, there is the canonical bijection  between the following sets:
		\begin{enumerate}
			\item $I_N^{\Gamma(M)}(E, P, Q)_{\mathrm{isog}}$,
			\item $I^{\Gamma(M)}_{N, \mathrm{mat}} \coloneqq \Gamma(M)\backslash\left\{ A \in  \mathrm{M}_2({\Z}) \relmiddle| \det A =N, A \equiv 
			\begin{pmatrix} 1 & 0 \\ 0 & 1 \end{pmatrix} \bmod M
			\right\}$.
		\end{enumerate}
		Same bijection induces the bijection between the following sets:
		\begin{enumerate}
			\item $P_N^{\Gamma(M)}(E, P, Q)_{\mathrm{isog}}$,
			\item $P^{\Gamma(M)}_{N, \mathrm{mat}} \coloneqq \{ [A] \in I^{\Gamma(M)}_{N, \mathrm{mat}} \mid A \text{ primitive} \}$.
		\end{enumerate}
	\end{enumerate}
\end{lemma}

\begin{proof}
	For \ref{item:lem:ellmatcorres1}, we may write $(E, C) = (E_{\tau}, \langle \frac{1}{M} + \Lambda_{\tau} \rangle)$ for some $\tau \in \H$.
	There is a bijection between
	\[
	\mathrm{SL}_2(\Z)\backslash\left\{ A \in  \mathrm{M}_2({\Z}) \relmiddle| \det A =N \right\}
	\]
	and
	\[
	\{ (f, E_1) \mid f \colon E_1\to E \text{ isogeny} \}/\mathrm{isom}
	\]
	by sending the equivalence class of $A= \begin{pmatrix}
	a & b \\ c & d
	\end{pmatrix} $ to the equivalence class of a pair of an elliptic curve $E_{A(\tau)}$ and an isogeny $f \colon E_{A(\tau)} \to E, z+\Lambda_{A(\tau)} \mapsto (c\tau +d)z + \Lambda_{\tau}$.
	
	Thus for $[f, E_1, C_1] \in I_N^{\Gamma_0(M)}(E, C)_{\mathrm{isog}}$, we obtain a matrix
	$A= \begin{pmatrix}
	a & b \\ c & d
	\end{pmatrix}$ of determinant $ N $ corresponding to $[f, E_1]$. We may write $E_1= E_{A(\tau)}, f(z + \Lambda_{A(\tau)}) = (c\tau +d)z + \Lambda_{\tau}$.
	Let $C_1 = \langle \frac{m(a\tau +b) + n(c\tau +d)}{M} + \Lambda_{A(\tau)} \rangle$. Since $C_1$ is a cyclic group of order $M$, $(m, n, M)=1$.
	By Lemma \ref{lem:elem_int}, we can assume that $(m, n)=1$.
	Then, we obtain a matrix
	$\gamma = \begin{pmatrix}
	k & l \\ m & n
	\end{pmatrix}$ for some integers $k, l$.
	The equivalence class of $[E_1, C_1]$ is
	\[
	\left[ E_{A(\tau)}, \left\langle \frac{m \cdot A(\tau) +n}{M} + \Lambda_{A(\tau)} \right\rangle \right]
	\]
	\[
	= \left[ E_{A(\tau)}, \left\langle \frac{m \cdot A(\tau) +n}{M} +\Z(k\cdot A(\tau) +l)+\Z(m\cdot A(\tau) +n) \right\rangle \right]
	\]
	\[
	= \left[ E_{\gamma A(\tau)}, \left\langle \frac{1}{M} + \Lambda_{\gamma A(\tau)} \right\rangle \right].
	\]
	The image of $\left\langle \frac{1}{M} + \Lambda_{\gamma A(\tau)} \right\rangle$ under $E_{\gamma A(\tau)} \xrightarrow{\sim} E_1 \to E$ is $\langle \frac{c\tau +d}{M} + \Lambda_{\tau} \rangle)$.
	Since this has to be $\langle \frac{1}{M} + \Lambda_{\tau} \rangle)$, we see $c \equiv 0 \bmod M, (d, M)=1$.
	Thus we obtain a matrix $A= \begin{pmatrix}
	a & b \\ c & d
	\end{pmatrix}$ such that $\det A =M, c \equiv 0 \bmod M, (d, M)=1$ corresponding to $[f, E_1, C_1]$.
	
	Suppose that the other matrix $A^\prime = \begin{pmatrix}
	a^\prime & b^\prime \\ c^\prime & d^\prime
	\end{pmatrix}$ also corresponds to $[f, E_1, C_1]$.
	Then there exists a matrix
	$\gamma = \begin{pmatrix}
	k & l \\ m & n
	\end{pmatrix}$ such that
	\[
	\gamma A = A^\prime, \quad
	\left\langle \frac{(ma+nc)\tau +(mb+nd)}{M} + \Lambda_{\tau} \right\rangle
	= \left\langle \frac{c^\prime \tau +d^\prime}{M} + \Lambda_{\tau} \right\rangle. 
	\]
	The second condition implies $m \equiv 0 \bmod M$.
	Thus $\gamma \in \Gamma_0(M)$.
	Therefore we obtain the  bijection between $I_N^{\Gamma_0(M)}(E, Q)_{\mathrm{isog}}$ and $I^{\Gamma_0(M)}_{N, \mathrm{mat}}$.
	
	This bijection induces a bijection between $P_N^{\Gamma_0(M)}(E, Q)_{\mathrm{isog}}$ and $P^{\Gamma_0(M)}_{N, \mathrm{mat}}$ since cyclic isogenies correspond to primitive matrices.
	
	For \ref{item:lem:ellmatcorres2}, we obtain the results by the same arguments for \ref{item:lem:ellmatcorres1}.
	
	For \ref{item:lem:ellmatcorres3}, We may write $(E, P, Q) = (E_{\tau}, \frac{\tau}{M} + \Lambda_{\tau}, \frac{1}{M} + \Lambda_{\tau}), (E_1, P_1, Q_1) = (E_{A(\tau)}, \frac{A(\tau)}{M} + \Lambda_{A(\tau)}, \frac{1}{M} + \Lambda_{A(\tau)})$ by the same arguments for \ref{item:lem:ellmatcorres1}.
	Since $P_1$ and $Q_1$ are sent to $P$ and $Q$ respectively, 
	$A \equiv \begin{pmatrix}
	1 & 0 \\ 0 & 1
	\end{pmatrix} \bmod M$.
	Therefore we obtain the results.
\end{proof}

\begin{proposition}\label{prop:irr_factrize}
	\begin{enumerate}
		\item\label{item:prop:im_zero1} For elliptic curves $(E, C), (E^\prime, C^\prime)$ with level structure for $\Gamma_0(M)$,
		\[
		\Phi_N^{\Gamma_0(M)}(t(E, C), t(E^\prime, C^\prime))
		= \prod_{N_1^2 \mid N, (M, N_1)=1} \Psi_{N/N_1^2}^{\Gamma_0(M)}(t(E, C), t(E^\prime, C^\prime)).
		\]
		\item\label{item:prop:im_zero2} For elliptic curves $(E, Q), (E^\prime, Q^\prime)$ with level structure for $\Gamma_1(M)$,
		\[
		\Phi_N^{\Gamma_0(M)}(t(E, Q), t(E^\prime, Q^\prime))
		= \prod_{N_1^2 \mid N, (M, N_1)=1} \Psi_{N/N_1^2}^{\Gamma_0(M)}(t(E, Q), t(E^\prime, Q^\prime)).
		\]
		\item\label{item:prop:im_zero3} For elliptic curves $(E, P, Q), (E^\prime, P^\prime, Q^\prime)$ with level structure for $\Gamma(M)$,
		\[
		\Phi_N^{\Gamma_0(M)}(t(E, P, Q), t(E^\prime, P^\prime, Q^\prime))
		= \prod_{N_1^2 \mid N, (M, N_1)=1} \Psi_{N/N_1^2}^{\Gamma_0(M)}(t(E, P, Q), t(E^\prime, P^\prime, Q^\prime)).
		\]
	\end{enumerate}
\end{proposition}

\begin{proof}
	The following map is bijective:
	\[
	\begin{array}{ccc}
	I^{\Gamma_0(M)}_{N, \mathrm{mat}}& \longrightarrow & \displaystyle\coprod_{N_1^2 \mid N, (M, N_1)=1} P^{\Gamma_0(M)}_{N/N_1^2, \mathrm{mat}} \\
	A & \longmapsto & \displaystyle\frac{1}{e(A)}A
	\end{array}
	\]
	where $e(A)$ is the content of $A$, that is the greatest common divisor of components in $ A $.
	Thus we obtain \ref{item:prop:im_zero1}.
	
	We can prove other statements by same arguments.
\end{proof}

\begin{lemma}\label{lem:transive}
	Define right group actions of $\Gamma_0(M), \Gamma_1(M)$ and $\Gamma(M)$ on $P^{\Gamma_0(M)}_{N, \mathrm{mat}}$, $P^{\Gamma_1(M)}_{N, \mathrm{mat}}$ and $P^{M}_{N, \mathrm{mat}}$ by matrix multiplication from the right respectively. Then, these actions are transitive.
\end{lemma}

\begin{proof}
	Let $\alpha \coloneqq \Gamma_0(M)
	\begin{pmatrix} N & 0 \\ 0 & 1 \end{pmatrix} \in P^{\Gamma_0(M)}_{N, \mathrm{mat}}$.
	For $\gamma \in \Gamma_0(M)$, the condition $\gamma$ stabilize $\alpha$ is equivalent to the condition $
	\begin{pmatrix} N & 0 \\ 0 & 1 \end{pmatrix}
	\gamma
	\begin{pmatrix} N^{-1} & 0 \\ 0 & 1 \end{pmatrix}
	\in \Gamma_0(M)$.
	Thus the stabilizer subgroup of $\Gamma_0(M)$ with respect to $\alpha$ is $\Gamma_0(M)_\alpha = \Gamma_0^0(M, N)$.
	Since $\# P^{\Gamma_0(M)}_{N, \mathrm{mat}} =[ \Gamma_0(M) : \Gamma_0^0(M, N) ]$, the injection $\Gamma_0(M)_\alpha \backslash \Gamma_0(M) \xrightarrow{\sim} \Gamma_0(M)\alpha \subset P^{\Gamma_0(M)}_{N, \mathrm{mat}}$ is bijective.
	Thus the action of $\Gamma_0(M)$ on $P^{\Gamma_0(M)}_{N, \mathrm{mat}}$ is transitive.
	
	For $\Gamma(M)$, if $P^{M}_{N, \mathrm{mat}} \neq \emptyset$, then $N \equiv 1 \bmod M$ and $\Gamma_0(M)
	\begin{pmatrix} N & 0 \\ 0 & 1 \end{pmatrix} \in P^{M}_{N, \mathrm{mat}}$.
	Thus same argument also works for $\Gamma_1(M), \Gamma(M)$.
\end{proof}

For the case when $Y=Y_0^0(M, N)$ and $(M, N)=1$, we can give complete systems of the sets of equivalence classes in Lemma \ref{lem:ellmatcorres}.

\begin{lemma}\label{lem:mat_rep}
	If $(M, N)=1$, then
	\[
	I^{\Gamma_0(M)}_{N, \mathrm{mat}} \coloneqq \left\{
	\begin{pmatrix} a & b \\ 0 & d \end{pmatrix} \in  \mathrm{M}_2({\Z}) \relmiddle| ad=N, 0 \le b < d \right\},
	\]
	\[
	P^{\Gamma_0(M)}_{N, \mathrm{mat}} \coloneqq \left\{
	\begin{pmatrix} a & b \\ 0 & d \end{pmatrix} \in I^{\Gamma_0(M)}_{N, \mathrm{mat}} \relmiddle| (a, b, d)=1 \right\}
	\]
	are complete systems of $I^{\Gamma_0(M)}_{N, \mathrm{mat}}, P^{\Gamma_0(M)}_{N, \mathrm{mat}}$ respectively.
\end{lemma}

\begin{proof}
	Since any two elements of $I^{\Gamma_0(M)}_{N, \mathrm{mat}}, P^{\Gamma_0(M)}_{N, \mathrm{mat}}$ are not $\Gamma_0(M)$-equivalent, it is enough to show that $\#I^{\Gamma_0(M)}_{N, \mathrm{mat}} = \#I^{\Gamma_0(M)}_{N, \mathrm{mat}}, \#P^{\Gamma_0(M)}_{N, \mathrm{mat}} = \#P^{\Gamma_0(M)}_{N, \mathrm{mat}}$.
	
	For $P^{\Gamma_0(M)}_{N, \mathrm{mat}}$,
	\[
	\begin{split}
	\# P^{\Gamma_0(M)}_{N, \mathrm{mat}}
	&=\sum_{d \mid N} \#\left\{b \in \Z/d\Z \relmiddle|\left(b, d, \dfrac{N}{d}\right)=1 \right\} \\
	&=\sum_{d \mid N} \dfrac{d}{\left(d, \dfrac{N}{d}\right)}\phi\left(\left(d, \dfrac{N}{d}\right)\right). 
	\end{split}
	\]
	Thus, if $(N_1, N_2)=1$, then $\# P^{\Gamma_0(M)}_{N_1, \mathrm{mat}} \# P^{\Gamma_0(M)}_{N_2, \mathrm{mat}} = \# P^{\Gamma_0(M)}_{N_1 N_2, \mathrm{mat}}$.
	For a prime $p$ and a positive integer $e$,
	\[
	\begin{split}
	\# P^{\Gamma_0(M)}_{p^e, \mathrm{mat}}
	&=\sum_{i=0}^{e} p^i\dfrac{\phi((p^i, p^{e-i}))}{(p^i, p^{e-i})} \\
	&=1+\sum_{i=1}^{e-1} p^i(1-p^{-1}) + p^e \\
	&=p^e(1+p^{-1}).
	\end{split}
	\]
	Therefore, for all positive integer $N$ prime to $M$, 
	\[
	\# P^{\Gamma_0(M)}_{p^e, \mathrm{mat}}
	=N \prod_{p \mid N} (1+p^{-1}) 
	= [\mathrm{SL}_2({\Z}):\Gamma_0(N)].
	\]
	On the other hand, for an elliptic curve $(E, C)$ with level structure for $\Gamma_0(M)$, the set $P^{\Gamma_0(M)}_{N, \mathrm{isog}}(E, C)$ can be identified with the inverse image of $(E, C)$ under
	\[
	\begin{array}{ccc}
	Y_0^0(M, N) & \longrightarrow & Y_0(M) \\
	{[E, C, D]} & \longmapsto & [E/C, (D+C)/C]. 
	\end{array}
	\]
	Thus
	\[
	\# P^{\Gamma_0(M)}_{p^e, \mathrm{mat}}
	=\#P^{\Gamma_0(M)}_{N, \mathrm{isog}}
	= [\Gamma_0^0(M, N):\Gamma_0(N)].
	\]
	Since we assume that $(M, N)=1$, we obtain $ \#P^{\Gamma_0(M)}_{N, \mathrm{mat}} = \#P^{\Gamma_0(M)}_{N, \mathrm{mat}}$.
	
	For $\#I^{\Gamma_0(M)}_{N, \mathrm{mat}}$, since there are bijections $I^{\Gamma_0(M)}_{N, \mathrm{mat}} \cong \coprod_{N_1^2 \mid N, (M, N_1)=1} P^{\Gamma_0(M)}_{N/N_1^2, \mathrm{mat}}$ in the proof of Proposition \ref{prop:irr_factrize} and $I^{\Gamma_0(M)}_{N, \mathrm{mat}} \cong \coprod_{N_1^2 \mid N, (M, N_1)=1} P^{\Gamma_0(M)}_{N/N_1^2, \mathrm{mat}}$,
	we obtain $\#I^{\Gamma_0(M)}_{N, \mathrm{mat}} = \#I^{\Gamma_0(M)}_{N, \mathrm{mat}}$.
\end{proof}

To the end of this paper, we consider only for $\Gamma_0(M)$ and assume that $X_0(M)$ has genus zero.
Thus $M$ satisfies $ 1\le M \le 10$ or $ M= 12, 13, 16, 18$ or $ 25$.
Also we assume that $M \neq 1$ and $(M, N)=1$.

Since $X_0(M)$ has genus zero, there exists a meromorphic function $t$ on $X_0(M)$ which has a simple pole at $\infty$ with residue 1 and is holomorphic on $X_0(M) \setminus \{\infty\}$.
We choose a suitable Hauptmodul for $\Gamma_0(M)$ so that modular polynomials have integral coefficients.

In the subsection 3.1 of \cite{Mai}, it is given explicitly that a meromorphic function $\hat{t}_M$ on $X_0(M)$ with $ \mathrm{div}(\hat{t}_M) = 0 - \infty $ which is  holomorphic on $X_0(M) \setminus \{ 0 \}$.
We set $t \coloneqq \hat{t}_M^{-1}$.
This function $t$ can be written explicitly by the product of Dedekind eta function $\eta(\tau)$ and takes the form $ q^{-1} + c_0 q + c_1 q^2 + \cdots, q \coloneqq e^{2\pi\sqrt{-1}\tau} $ with integer coefficients $ c_i, i \ge 2 $.
For example, if $M-1 \mid 24$, then
\[
t(\tau)= \left( \frac{\eta(\tau)}{\eta(M\tau)} \right)^{\frac{24}{M-1}}.
\]
The eta-product expression of $t$ for all $M$ which we treat are listed in Table \ref{tab:Haupt}.

\begin{table}[h]
	\caption{Hauptmodul $t$ for $\Gamma_0(M)$.}
	\label{tab:Haupt}
	\begin{minipage}[t]{.45\textwidth}
		\begin{tabular}{lll}
			\hline\noalign{\smallskip}
			$M$ & $t$ \\
			\noalign{\smallskip}\hline\noalign{\smallskip}
			2 & $ \dfrac{\eta(\tau)^{24}}{\eta(2\tau)^{24}} $ \\
			3 & $\dfrac{\eta(\tau)^{12}}{\eta(3\tau)^{12}}  $ \\
			4 & $\dfrac{\eta(\tau)^{8}}{\eta(4\tau)^{8}}  $ \\
			5 & $\dfrac{\eta(\tau)^{6}}{\eta(5\tau)^6} $ \\
			6 & $\dfrac{\eta(\tau)^{5}\eta(3\tau)}{\eta(2\tau)\eta(6\tau)^{5}}$ \\
			7 & $\dfrac{\eta(\tau)^{4}}{\eta(7\tau)^{4}}$ \\
			8 & $\dfrac{\eta(\tau)^{4}\eta(4\tau)^{2}}{\eta(2\tau)^{2}\eta(8\tau)^{4}} $ \\
			\noalign{\smallskip}\hline
		\end{tabular}
	\end{minipage}
	\hfill
	\begin{minipage}[t]{.45\textwidth}
		\begin{tabular}{lll}
			\hline\noalign{\smallskip}
			$M$ & $t$ \\		
			\noalign{\smallskip}\hline\noalign{\smallskip}	
			9 & $\dfrac{\eta(\tau)^{3}}{\eta(9\tau)^{3}} $ \\
			10 & $\dfrac{\eta(\tau)^{3}\eta(5\tau)}{\eta(2\tau)\eta(10\tau)^{3}}$ \\
			12 & $\dfrac{\eta(\tau)^{3}\eta(4\tau)\eta(6\tau)^{2}}{\eta(2\tau)^{2}\eta(3\tau)\eta(12\tau)^{3}} $ \\
			13 & $\dfrac{\eta(\tau)^{2}}{\eta(13\tau)^{2}} $ \\
			16 & $\dfrac{\eta(\tau)^{2}\eta(8\tau)}{\eta(2\tau)\eta(16\tau)^{2}} $ \\
			18 & $\dfrac{\eta(\tau)^{2}\eta(6\tau)\eta(9\tau)}{\eta(2\tau)\eta(3\tau)\eta(18\tau)^{2}} $ \\
			25 & $\dfrac{\eta(\tau)}{\eta(25\tau)}$ \\
			\noalign{\smallskip}\hline
		\end{tabular}
	\end{minipage}
\end{table}

We can prove that the modular polynomials defined by this Hauptmoduln $t$  are polynomials with integral coefficients.

\begin{theorem}\label{thm:int_coeff}
	If $(M, N)=1$, then $\Phi_{N}^{\Gamma_0(M)}, \Psi_{N}^{\Gamma_0(M)} \in \Z[X, Y]$.
	The degrees of these polynomials with respect to $X$ or $Y$ satisfy
	\[
	\deg_X \Phi_{N}^{\Gamma_0(M)} =\deg_Y \Phi_{N}^{\Gamma_0(M)}, \deg_X \Psi_{N}^{\Gamma_0(M)} =\deg_Y \Psi_{N}^{\Gamma_0(M)}=(\mathrm{SL}_2(\Z):\Gamma_0(M)).
	\]
\end{theorem}

\begin{proof}
	By Proposition \ref{prop:irr_factrize}, it is enough to show that $\Psi_{N}^{\Gamma_0(M)} \in \Z[X, Y]$ and $ \deg_X \Psi_{N}^{\Gamma_0(M)} =\deg_Y \Psi_{N}^{\Gamma_0(M)}$.
	Let $n \coloneqq \# P^{\Gamma_0(M)}_{N, \mathrm{mat}}$ and
	$\Psi_{N}^{\Gamma_0(M)}(X, t(\tau)) = X^n + k_1(\tau)X^{n-1} + \cdots + k_n(\tau)$.
	For $ \begin{pmatrix}
	a & b \\ 0 & d
	\end{pmatrix} \in P^{\Gamma_0(M)}_{N, \mathrm{mat}}$,
	$t(\frac{a\tau +b}{d}) $ is an element of $ \Z[\zeta_N][[q_N]][q_N^{-1}]$
	as a formal series, where $q_N \coloneqq e^{2\pi\sqrt{-1}\tau/N}$.
	Thus $k_i(\tau) \in \Z[\zeta_N][[q_N]][q_N^{-1}]$ for each $i=1, \cdots, n$.
	
	Let $\sigma \in (\Z/N\Z)^{\times} \cong \Gal(\Q(\zeta_N)/\Q)$.
	Since $\sigma$ induces a bijection from $P^{\Gamma_0(M)}_{N, \mathrm{mat}}$ to $P^{\Gamma_0(M)}_{N, \mathrm{mat}}$ by
	$\sigma \cdot
	\begin{pmatrix} a & b \\ 0 & d \end{pmatrix}
	\coloneqq
	\begin{pmatrix} a & \sigma b \\ 0 & d \end{pmatrix} $,
	\[
	\begin{split}
	\Psi_{N}^{\Gamma_0(M)}(X, t(\tau))
	&= \prod_{\bigl( \begin{smallmatrix} a & b \\ 0 & d \end{smallmatrix} \bigr) \in P^{\Gamma_0(M)}_{N, \mathrm{mat}}
	} \left(X - t\left(\dfrac{a\tau^\prime + \sigma b}{d}\right)\right) \\
	&= \prod_{\bigl( \begin{smallmatrix} a & b \\ 0 & d \end{smallmatrix} \bigr) \in P^{\Gamma_0(M)}_{N, \mathrm{mat}}} (X - t(\zeta_m^{\sigma ab} q_N^{a^2})) \\
	&= \sigma\left( \Psi_{N}^{\Gamma_0(M)}(X, t(\tau)) \right).
	\end{split}
	\]
	Thus $k_i(\tau) \in \Z[[q_N]][q_N^{-1}]$ for each $i=1, \cdots, n$.
	
	By Lemma \ref{lem:transive}, 
	$\Psi_{N}^{\Gamma_0(M)}(X, t(\gamma(\tau)))= \Psi_{N}^{\Gamma_0(M)}(X, t(\tau))$
	for all $\gamma \in \Gamma_0(M)$.
	Thus $k_i(\tau)$ is $\Gamma_0(M)$-invariant.
	Especially, $k_i(\tau+1) = k_i(\tau)$.
	Thus $k_i(\tau)$ is a function of $q=e^{2\pi\sqrt{-1}\tau}$.
	Then we obtain $k_i(\tau) \in \Z[[q]][q^{-1}]$.
	
	Write $k_i = a_r q^{-r} + O(q^{-r+1})$ for some $r \ge 1, a_r \in \Z$.
	Then $k_i -a_r t^{r} = a_{r-1} q^{-r+1} + O(q^{-r+2})$ for some $a_{r-1} \in \Z$.
	By repeating this process, we obtain a polynomial $p_i(T) \in \Z[T]$ such that $k_i - p_i(t(\tau)) \in \Z[[q]]$.
	Since the function $k_i(\tau) - p_i(t(\tau))$ is a modular form of weight 0 with respect to $\Gamma_0(M)$,  $k_i - p_i \in \Z$.
	Thus $k_i \in \Z[t(\tau)]$.
	
	For each $i=1, \cdots, n$, $k_i(\tau)$ is written as $\Z$-linear combination of
	\[
	\{ t(A_1(\tau)) \cdots t(A_i(\tau)) \mid A_1, \cdots A_i \in P^{\Gamma_0(M)}_{N, \mathrm{mat}} \}.
	\]
	Thus $k_i(\tau)$ is zero or has degree $i$ as holomorphic function on $X_0(M)$.
	Denote $\Psi_{N}^{\Gamma_0(M)}(X, Y) = X^n + k_1(Y)X^{n-1} + \cdots k_n(Y)$.
	Then $k_i(Y) = 0 $ or $\deg k_i(Y) =i$.
	Since $k_n(Y) = (-1)^n \prod_{A \in P^{\Gamma_0(M)}_{N, \mathrm{mat}}} t(A(\tau))$ is nonzero, $\deg k_n(Y) = n$.
	Thus $\deg_X \Psi_{N}^{\Gamma_0(M)} =\deg_Y \Psi_{N}^{\Gamma_0(M)}$.
\end{proof}

By Theorem \ref{thm:int_coeff}, we can calculate modular polynomials for small $ M, N $.
Let $ n \coloneqq (\mathrm{SL}_2(\Z):\Gamma_0(M)) $ and $ \Psi_{N}^{\Gamma_0(M)}(X, Y) = \sum_{0 \le i, j \le n} a_{ij} X^i Y^j, a_{ij} \in \Z $.
By comparing coefficients of $ \prod_{A \in P^{\Gamma_0(M)}_{N, \mathrm{mat}}} (X - t(A(\tau))) $, we obtain $ a_{ij} $.
We list modular polynomials for small $M, N$ in Table \ref{tab:Haupt}.

\begin{table}
	\caption{Modular polynomials for small $M, N$.}
	\label{tab:mod_poly}
	\centering
	\begin{tabularx}{\linewidth}{llX} 
		\hline\noalign{\smallskip}
		$M$ & $ N $ & $\Psi_{N}^{\Gamma_0(M)}$ \\ 
		\noalign{\smallskip}\hline\noalign{\smallskip}
		2 & 3 & $
		{X}^{4} -{X}^{3}{Y}^{3} -72 {X}^{3}{Y}^{2} -900 {X}^{3}Y-72 {X}^{2}{Y}
		^{3} +28422 {X}^{2}{Y}^{2} -294912 {X}^{2}Y$  $ -900 X{Y}^{3} -294912 X{Y}
		^{2} -16777216 XY+{Y}^{4}
		$ \\
		2 & 5 & $
		{X}^{6}-{X}^{5}{Y}^{5} -226021320 {X}^{5}{Y}^{4} -256122862563480 {X}^{5}{Y}^{3} $  $-31298355995833670720 {X}^{5}{Y}^{2}- 
		954325239073593568474830 {X}^{5}Y $  $+480 {X}^{4}{Y}^{5} +6409763520 {X}^{4}{Y}^{4} +1506162672564480 {X}^{4}{Y}^{3} $  $+51392022681804939270 {X}^{4}{Y}^{2} +512792264635738861076480 {X}^{4}Y $  $-26280 {X}^{3}{Y}^{5} -27297299520 {X}^{3}{Y}^{4} -1142420172039180 {X}^{3}{Y}^{3} $  $-9253863460236165120 {X}^{3}{Y}^{2} -25782171526594906030080 {X}^{3}Y$  $+
		196480 {X}^{2}{Y}^{5} +13441732620 {X}^{2}{Y}^{4} +74539825889280 {X}^{2}{Y}^{3}$  $ +107537987083960320{X}^{2}{Y}^{2} +62128267366320046080 {X}^{2}Y -90630 X{Y}^{5}$  $ -201195520 X{Y}^{4} -73484206080 X{Y}^{3} -8246337208320 X{Y}^{2}$  $-281474976710656 XY+ {Y}^{6}
		$ \\
		3 & 2 & $ 
		{X}^{3} -{X}^{2}{Y}^{2} -24 {X}^{2}Y -24 X{Y}^{2} -729 XY +{Y}^{3}$ \\ 
		4 & 3 & $
		X^4 -X^3 Y^3 -24 X^3 Y^2 -132 X^3 Y -24 X^2 Y^3 -762 X^2 Y^2 $  $-6144 X^2 Y$  $ -132 XY^3 -6144 X Y^2 -65536 XY +Y^4$ \\
		5 & 2 & $
		{X}^{3} -{X}^{2}{Y}^{2} -12 {X}^{2}Y -12 X{Y}^{2} -125XY +{Y}^{3}
		$ \\
		5 & 3 & $
		{X}^{4} -{X}^{3}{Y}^{3} -18 {X}^{3}{Y}^{2} -81 {X}^{3}Y -18 {X}^{2}{Y}^	{3} - 414 {X}^{2}{Y}^{2} - 2250 {X}^{2}Y - 81 X{Y}^{3} $  $- 2250 X{Y}^{2} - 15625 XY+ {Y}^{4}$ \\
		\noalign{\smallskip}\hline
	\end{tabularx}
\end{table}

\begin{theorem}
	The polynomial $\Psi_{N}^{\Gamma_0(M)}$ is irreducible over $\C(t)$.
\end{theorem}

\begin{proof}
	The set of roots of $\Psi_{N}^{\Gamma_0(M)}$ over $\C(t)$ is the set
	$\{ t(A(\tau)) \mid A \in P^{\Gamma_0(M)}_{N, \mathrm{mat}} \}$
	of  holomorphic function on $\H$.
	The field $L \coloneqq \C(t)(t(A(\tau)) \mid A \in P^{\Gamma_0(M)}_{N, \mathrm{mat}})$ is the splitting field of $\Psi_{N}^{\Gamma_0(M)}$ over $\C(t)$.
	Consider the group homomorphism
	\[
	\begin{array}{cccc}
	\rho \colon & \Gamma_0(M) & \longrightarrow & \Gal(L/\C(t)) \\
	& \gamma & \longmapsto & (f \mapsto f \circ \gamma).
	\end{array}
	\]
	Since the fixed field of the image of $\rho$ is $\C(X_0(M)) = \C(t)$, $\rho$ is surjective.
	By Lemma \ref{lem:transive}, the action of $\Gal(L/\C(t))$ on the set of roots of $\Psi_{N}^{\Gamma_0(M)}$ is transitive
	Thus $\Psi_{N}^{\Gamma_0(M)}$ is irreducible.
\end{proof}

\begin{corollary}
	The function field of compact Riemann surface $X_0^0(M, N) \coloneqq \Gamma_0^0(M, N)\backslash \H \cup \Q \cup \{ \infty \}$ is
	\[
	\C(X_0^0(M, N)) = \C(t, t_N)
	\]
	where $ t_N(\tau) \coloneqq t(N\tau) $.
\end{corollary}

\begin{proof}
	Since $\Psi_{N}^{\Gamma_0(M)}$ is the minimal polynomial of $t_N$ over $\C(t)$, 
	\[
	\begin{split}
	& [ \C(t, t_N) : \C(t) ]
	= \deg \#\Psi_{N}^{\Gamma_0(M)}(X, t)
	= \# P^{\Gamma_0(M)}_{N, \mathrm{mat}} \\
	&= [\Gamma_0(M):\Gamma_0^0(M, N)]
	= [ \C(X_0^0(M, N)) : \C(t(\tau)) ].
	\end{split}
	\]
	Thus $\C(X_0^0(M, N)) = \C(t, t_N)$.
\end{proof}

\begin{proposition}
	If $ (M, N)=1 $, then the modular polynomials $\Phi_{N}^{\Gamma_0(M)}$ and $ \Psi_{N}^{\Gamma_0(M)}$ are symmetric.
\end{proposition}

\begin{proof}
	For integers $ a, b $ such that $ Na - Mb =1 $, define Atkin-Lehner involution $ W_N $ by
	\[
	W_N \coloneqq \frac{1}{\sqrt{N}}
	\begin{pmatrix} Na & b \\ MN & N \end{pmatrix}
	\in \mathrm{SL}_2(\R).
	\]
	The matrix $ W_N $ defines a well-defined involution on $ X_0^0(M, N) $.
	The calculation
	\[
	\frac{1}{\sqrt{N}}
	\begin{pmatrix} N & 0 \\ 0 & 1 \end{pmatrix}
	W_N =
	\begin{pmatrix} Na & b \\ M & 1 \end{pmatrix}
	\in \Gamma_0(M)
	\]
	shows that $ t_N \circ W_N =t $.
	Since $ W_N^2 = \mathrm{id} $, $ t \circ W_N =t_N $.
	Thus we obtain $\Psi_{N}^{\Gamma_0(M)}(t_N, t)=0$.
	Since $\Psi_{N}^{\Gamma_0(M)}(X, t)=0$ and $\Psi_{N}^{\Gamma_0(M)}(t, X)=0$ has same degree, $\Psi_{N}^{\Gamma_0(M)}(t, X)=0$ is also the minimal polynomial of $ t_N $.
	Thus we obtain  $\Psi_{N}^{\Gamma_0(M)}(X, t) = \Psi_{N}^{\Gamma_0(M)}(t, X)=0$.
\end{proof}

\section{Quadratic forms}\label{sec:quad_form}

In this Section, we study some classes of primitive positive definite binary quadratic forms over $ \Z $ in order to prove Theorem \ref{thm:main}.

For a quadratic form $Q$ on a free $\Z$-module $L$ of rank 2, the associated bilinear form $(\ ,\ )$ on $L$ is defined by
\[
(x, y) \coloneqq Q(x+y) - Q(x) - Q(y).
\]
The determinant of $Q$ is the determinant of the matrix representation of the associated bilinear form.
It is denoted $\det Q$ or $\deg L$.
We set the stabilizer subgroup with respect to $ Q $ by
\[
\mathrm{SL}_2(\Z)_Q \coloneqq \{ A \in \mathrm{SL}_2(\Z) \mid Q \circ A = Q \},
\Gamma_0(M)_Q \coloneqq \{ A \in \Gamma_0(M) \mid Q \circ A = Q \}.
\]
These groups are finite.
For a fixed $\Z$-basis $e_1, e_2$ of $L$, we denote $Q=[a, b, c]$ if $Q(ke_1 + le_2) = ak^2 + bkl +cl^2$ for $k, l \in \Z$. 

We omit a proof of the following lemma.

\begin{lemma}\label{lem:e_D_corresp}
	For positive integers $ e, D $ such that $ e^2 \mid D $, there is a bijection between 
	\[
	\left\{ (d, [a, b, c], Z) \relmiddle| d^2 \mid D, Z \mid (e, d), \left( \frac{M}{(a, M)}\right)^2 (4ac-b^2)= \frac{D}{d^2}, [a, b, c] \text{ primitive } \right\}
	\]
	and
	\[
	\left\{ ([Ma, b, c], Z) \relmiddle| Z \mid e, 4Mac-b^2= \frac{D}{Z^2} \right\}
	\]
	by sending $ (d, [a, b, c], Z) $ to $ (\frac{M}{(a, M)} \frac{d}{Z}[a, b, c], Z) $ and $ ([Ma, b, c], Z) $ to \\
	$ ((Ma, b, c)aZ, \frac{1}{(Ma, b, c)}[a, b, c], Z) $.
\end{lemma}

For positive integers $ D $ and $ g \mid M $, we define $ H^M(D) $ by the following sum with respect to $\Gamma_0(M)$-equivalence classes of positive definite binary quadratic forms $ Q $ such that $Q= [Ma, b, c] $ for some integers $ a, b, c $ and whose determinant is $ D $ :
\[
H^M(D) \coloneqq \sum_{Q= [Ma, b, c],\ \det Q =D}
\frac{1}{[\Gamma_0(M)_Q : \{\pm1 \}]}.
\]
In the definition in $ H^M(D) $, we refer to Lemma 2 in Section 1.1 in \cite{HZ}.

We prepare the following lemma in order to express $ H^M(D)  $ by $ H(D) $, which is introduced in Proposition \ref{prop:Hurwitz}.

\begin{lemma}\label{lem:class_num}
	For a prime $ p $ and a positive integer $ D $ with $ p^2 | D $, 
	\[
	H \left( \frac{D}{p^2} \right)
	= \frac{1}{1+ \frac{1}{p^v} \frac{1-p^v}{1-p} \left( 1-\frac{\chi_{D}(p)}{p} \right) }
	H(D)
	\]
	where $ v \coloneqq \lfloor \frac{\mathrm{ord}_p D}{2} \rfloor $.
\end{lemma}

\begin{proof}
	Since all square divisors of $ \frac{D}{p^{2v}} $ are prime to $ p $,
	\[
	H (D)
	= \sum_{ f^2 \mid D/p^{2v}} \left( h\left( \frac{D}{f^2} \right) + h\left( \frac{D}{f^2 p^2} \right) + \cdots + h\left( \frac{D}{f^2 p^{2v}} \right) \right)
	\]
	where $h(D)$ is the number of $\mathrm{SL}_2(\Z)$-equivalence classes of primitive positive definite binary quadratic forms over $\Z$ with determinant $D$, counting the forms equivalent to $ex_1^2 + ex_2^2$ and $ex_1^2 + ex_1 x_2 + ex_2^2$ with multiplicity $1/2$ and $1/3$ respectively. 
	By Corollary 7.28 in \cite{Cox}, this is
	\[
	\begin{split}	
	& \sum_{ f^2 \mid D/p^{2v}} \left( 1+ \left( 1- \frac{\chi_{D}(p)}{p} \right)^{-1}
	\left( \frac{1}{p} + \cdots + \frac{1}{p^v} \right) \right) h\left( \frac{D}{f^2} \right) \\
	&= \left( 1- \frac{\chi_{D}(p)}{p} \right)^{-1} \left( \frac{1}{p} + \cdots + \frac{1}{p^v} +1 - \frac{\chi_{D}(p)}{p} \right) \sum_{ f^2 \mid D/p^2}  h\left( \frac{D}{f^2} \right). 
	\end{split}
	\]
	Thus 
	\[
	\begin{split}
	H \left( \frac{D}{p^2} \right)
	&= \sum_{ f^2 \mid D/p^{2v}} \left( h\left( \frac{D}{f^2 p^2} \right) + \cdots + h\left( \frac{D}{f^2 p^{2v}} \right) \right) \\
	&= \left( 1- \frac{\chi_{D}(p)}{p} \right)^{-1} \left( \frac{1}{p} + \cdots + \frac{1}{p^v} \right) \sum_{ f^2 \mid D/p^2}  h\left( \frac{D}{f^2} \right) \\
	&= \frac{  \frac{1}{p} + \cdots + \frac{1}{p^v} }{ \frac{1}{p} + \cdots + \frac{1}{p^v} +1 - \frac{\chi_{D}(p)}{p}}
	H(D) \\
	&= \frac{1}{1+ \frac{1}{p^v} \frac{1- p^v}{1-p} \left( 1 - \frac{\chi_{D}(p)}{p} \right) }
	H(D).
	\end{split}
	\]
\end{proof}

We can prove the following proposition by the above lemma.

\begin{proposition}\label{prop:sum_of_H^M_and_H}
	Let $ p $ be a prime such that $ X_0(p) $ has genus zero, that is, $ p=2, 3, 5, 7, 13 $.
	Let $ e, D $ be positive integers such that $ D \equiv 0$ or $ D \equiv -1 \bmod 4$ , $p \nmid e $ and $ e^2 \mid D $. Then
	\[
	\sum_{d \mid e,\ d^2 \mid D} d \cdot H^p \left( \frac{D}{p^2} \right)
	= A^p(D) \sum_{d \mid e,\ d^2 \mid D}
	d \cdot H \left( \frac{D}{d^2} \right)
	\]
	where
	\[
	A^p(D) \coloneqq
	\begin{cases}
	1 + \chi_{D}(p) & \text{ if } v \coloneqq \lfloor \frac{\mathrm{ord}_p D}{2} \rfloor =0 \\
	1+ \frac{p}{1+ \frac{1}{p^v} \frac{1-p^v}{1-p} \left( 1-\frac{\chi_{D}(p)}{p} \right) } & \text{ if } v \ge 1
	\end{cases}
	\]
	and $ \chi_{D} $ is the quadratic character associated to the imaginary quadratic field $ \Q(\sqrt{-D}) $.
\end{proposition}

\begin{proof}
	By Lemma 3.2 in \cite{CK}, if there exists an integer $ h \bmod 2p $ with $ h^2 \equiv -D \bmod 4p $, then
	\[
	\sum_{Q= [Ma, b, c],\ \det Q =D,\ b \equiv h \bmod 2p}
	\frac{1}{[\Gamma_0(M)_Q : \{\pm1 \}]}
	= H(D) + p \cdot H \left( \frac{D}{p^2} \right)
	\]
	where the left side is the sum with respect to $\Gamma_0(M)$-equivalence classes of positive definite binary quadratic forms $ Q $ and $ H \left( \frac{D}{p^2} \right) \coloneqq 0 $ if $ p^2 \nmid D $.
	Thus in the situation of this proposition,
	\[
	H^p(D)
	= n_D(p) \left(  H(D) + p \cdot H \left( \frac{D}{d^2} \right)\right)
	\]
	where $ n_D(p) \coloneqq \# \{ h \bmod 2p \mid  h^2 \equiv -D \bmod 4p \} $.
	For the case when $ p \neq 2 $, $ n_D(p) = 1 + \left( \frac{-D}{p} \right) $ because $ -D $ is a square modulo 4 by assumption.
	If $ p=2 $, then
	\[
	\begin{split}
	n_D(p) &= 
	\begin{cases}
	0 & \text{ if } -D \equiv 5 \bmod 8 \\
	1 & \text{ if } -D \equiv 0 \bmod 4 \\
	2 & \text{ if } -D \equiv 1 \bmod 8  
	\end{cases} \\
	&= 1+ \chi_{D}(2).
	\end{split}
	\]
	
	Suppose $ \mathrm{ord}_p D \le 1 $. Then for an integer $ d $ with $ d \mid e$ and $ d^2 \mid D $, $ H^p \left( \frac{D}{d^2} \right) = (1 + \chi_{D}(p)) H \left( \frac{D}{d^2} \right) $ since $ d $ is prime to $ p $ and $  \left( \frac{-D}{p} \right) =  \chi_{D}(p) $. Thus 
	\[
	\sum_{d \mid e,\ d^2 \mid D} d \cdot H^p \left( \frac{D}{d^2} \right)
	= (1 + \chi_{D}(p)) \sum_{d \mid e,\ d^2 \mid D}
	d \cdot H \left( \frac{D}{d^2} \right).
	\]
	
	Suppose $ \mathrm{ord}_p D \ge 2 $. Then, by Lemma \ref{lem:class_num}, 
	\[
	\sum_{d \mid e,\ d^2 \mid D} d \cdot H^p \left( \frac{D}{d^2} \right)
	= A^p(D) \sum_{d \mid e,\ d^2 \mid D}
	d \cdot H \left( \frac{D}{d^2} \right)
	\]
	since $ \lfloor \frac{\mathrm{ord}_p D}{2} \rfloor = \lfloor \frac{\mathrm{ord}_p (D/d^2)}{2} \rfloor $.
\end{proof}

\section{Intersection numbers}\label{sec:intnum}

In this Section, we calculate the intersection number of modular polynomials.

\subsection{Intersection multiplicity and elliptic curves with level structure for $\Gamma_0(M)$}

In this subsection, we express the intersection multiplicity of modular polynomials by a sum with respect to isogenies of elliptic curves with level structure for $\Gamma_0(M)$.

For $\tau \in H$, let $E_{\tau}^{\Gamma_0(M)} \coloneqq [E_\tau, \langle \frac{1}{M} + \Lambda_{\tau} \rangle]$ be an elliptic curve with level structure for $\Gamma_0(M)$.

For $\tau, \tau^\prime \in H$, denote
\[
\Hom(E_{\tau}^{\Gamma_0(M)}, E_{\tau^\prime}^{\Gamma_0(M)})
\coloneqq \left\{ f \colon E_{\tau} \to E_{\tau^\prime} \relmiddle| f \text{ isogeny}, f\left(\left\langle \frac{1}{M} + \Lambda_{\tau} \right\rangle\right) \subset \left\langle \frac{1}{M} + \Lambda_{\tau^\prime} \right\rangle \right\},
\]
\[
\End(E_{\tau}^{\Gamma_0(M)}) \coloneqq \Hom(E_{\tau}^{\Gamma_0(M)}, E_{\tau}^{\Gamma_0(M)}),
\quad d(E_{\tau}^{\Gamma_0(M)}) \coloneqq \disc \End(E_{\tau}^{\Gamma_0(M)}),
\]
\[
\Hom^{\Gamma_0(M)}(E_{\tau}^{\Gamma_0(M)}, E_{\tau^\prime}^{\Gamma_0(M)})
\coloneqq \left\{ f \in \Hom(E_{\tau}^{\Gamma_0(M)}, E_{\tau^\prime}^{\Gamma_0(M)}) \relmiddle|  f\left(\left\langle \frac{1}{M} + \Lambda_{\tau} \right\rangle\right)= \left\langle \frac{1}{M} + \Lambda_{\tau^\prime} \right\rangle \right\}.
\]
Remark that $\Hom^{\Gamma_0(M)}(E_{\tau}^{\Gamma_0(M)}, E_{\tau^\prime}^{\Gamma_0(M)})$ is not a $\Z$-module.
For any $f \in \Hom(E_{\tau}^{\Gamma_0(M)}, E_{\tau^\prime}^{\Gamma_0(M)})$, its dual isogeny $\hat{f}$ is an element of $\Hom(E_{\tau^\prime}^{\Gamma_0(M)}, E_{\tau}^{\Gamma_0(M)})$.
However, for $f \in \Hom^{\Gamma_0(M)}(E_{\tau}^{\Gamma_0(M)}, E_{\tau^\prime}^{\Gamma_0(M)})$, its dual isogeny $\hat{f}$ is not always an element of $\Hom^{\Gamma_0(M)}(E_{\tau^\prime}^{\Gamma_0(M)}, E_{\tau}^{\Gamma_0(M)})$.

For $\tau, \tau^\prime \in H$, define a map $\deg \colon \Hom(E_{\tau}^{\Gamma_0(M)}, E_{\tau^\prime}^{\Gamma_0(M)}) \to \Z_{\ge 0}$ which sends an isogeny to its degree. This map is a positive definite quadratic form on the free $\Z$-module $\Hom^{\Gamma_0(M)}(E_{\tau}^{\Gamma_0(M)}, E_{\tau^\prime}^{\Gamma_0(M)})$.
Thus we can consider $ \det (\Hom^{\Gamma_0(M)}(E_{\tau}^{\Gamma_0(M)}, E_{\tau^\prime}^{\Gamma_0(M)})) $.
We denote $ d(E_{\tau}^{\Gamma_0(M)}) \coloneqq \det (\End(E_{\tau}^{\Gamma_0(M)})) $.

\begin{lemma}\label{lem:isog_ell_lev}
	Let $\tau \in H$ be imaginary quadratic and $a, b, c \in \Z$.
	\begin{enumerate}
		\item $\End(E_{\tau}^{\Gamma_0(M)}) = \Z[\mathrm{lcm}(a, M)\tau], d(E_{\tau}^{\Gamma_0(M)}) = (\frac{M}{(a, M)})^2 (b^2 -4ac)$.
		\item For $\alpha \in \C \setminus \Z$, the followings are equivalent:
		\begin{enumerate}
			\item  \label{item:lem:isog_ell_lev1}$\alpha \in \Z[\mathrm{lcm}(a, M)\tau]$.
			\item \label{item:lem:isog_ell_lev2} $\alpha$ is an imaginary quadratic integer and $d^2 \cdot d(E_{\tau}^{\Gamma_0(M)}) = \trace(\alpha)^2 -4\norm(\alpha)$ for some $d \in \Z$.
		\end{enumerate}
	\end{enumerate}
\end{lemma}

\begin{proof}
	Since $\End(E_{\tau}) = \Z[a\tau]$,
	\[
	\begin{split}
	\End(E_{\tau}^{\Gamma_0(M)})
	&= \left\{ f \in \End(E_{\tau}) \relmiddle| f\left(\frac{1}{M} + \Lambda_{\tau}\right) \in \left\langle \frac{1}{M} + \Lambda_{\tau} \right\rangle \right\} \\
	&\cong \left\{ \alpha \in \Z[a\tau] \relmiddle| \frac{\alpha}{M} + \Lambda_{\tau} \in \left\langle \frac{1}{M} + \Lambda_{\tau} \right\rangle \right\} \\
	&= \{ m+na\tau \mid  m, n \in \Z, M \mid na \} \\
	&= \Z[\mathrm{lcm}(a, M)\tau].
	\end{split}
	\]
	Since the minimal polynomial of $\mathrm{lcm}(a, M)\tau$ is 
	$X^2 + \frac{M}{(a, M)}bX +(\frac{M}{(a, M)})^2ac$,  the discriminant of $\Z[\mathrm{lcm}(a, M)\tau]$ is $(\frac{M}{(a, M)})^2 (b^2 -4ac)$.
	
	Let $\alpha \in \Z[\mathrm{lcm}(a, M)\tau] \setminus \Z$.
	A sublattice $\Z + \Z\alpha$ of $ \End(E_{\tau}^{\Gamma_0(M)}) $ has determinant $ 4\norm(\alpha) - \trace(\alpha)^2$.
	Especially, $ d(E_{\tau}^{\Gamma_0(M)}) = \det (\End(E_{\tau}^{\Gamma_0(M)})) $.
	Moreover, $d^2 \cdot \det (\End(E_{\tau}^{\Gamma_0(M)})) =\det(\Z + \Z\alpha)$ for some $d \in \Z$.
	
	Conversely, if $\alpha \in \C \setminus \Z$ satisfies \ref{item:lem:isog_ell_lev2}, then
	\[
	\alpha = \frac{\trace(\alpha) \pm d\sqrt{d(E_{\tau}^{\Gamma_0(M)})}}{2}.
	\]
	On the other hand,
	\[
	\mathrm{lcm}(a, M)\tau = \frac{1}{2}\left( -\frac{bM}{(a, M)} + \sqrt{d(E_{\tau}^{\Gamma_0(M)})} \right).
	\]
	Thus 
	\[
	\alpha = \frac{1}{2}\left( \trace(\alpha) \pm \frac{M}{(a, M)}bd \right) \pm d \cdot \mathrm{lcm}(a, M)\tau.
	\]
	Since $d^2 \cdot (b^2-4ac) = \trace(\alpha)^2 -4\norm(\alpha)$,
	\[
	\frac{1}{2}\left( \trace(\alpha) + \frac{M}{(a, M)}bd \right)
	\cdot \frac{1}{2}\left( \trace(\alpha) - \frac{M}{(a, M)}bd \right)
	= \norm(\alpha) - \left(\frac{M}{(a, M)}d\right)^2ac.
	\]
	Therefore $ \frac{1}{2}\left( \trace(\alpha) \pm \frac{M}{(a, M)}bd \right) \in \Z $ and $ \alpha \in \Z[\mathrm{lcm}(a, M)\tau]$.
\end{proof}

For a positive integer $N$, $ T_{N}^{\Gamma_0(M)}$ denotes the
affine plane algebraic curve defined by the modular polynomial $\Phi_{N}^{\Gamma_0(M)}$.

\begin{proposition}\label{prop:isog_disc}
	Let $N_1, N_2$ be positive integers prime to $M$.
	Then, the curves $T_{N_1}^{\Gamma_0(M)} $ and $T_{N_2}^{\Gamma_0(M)} $ intersect properly if and only if $N_1 N_2 $ is not a square.
	
	If $T_{N_1}^{\Gamma_0(M)} $ and $T_{N_2}^{\Gamma_0(M)} $ intersect properly, for any $ (t(\tau), t(\tau^\prime)) \in T_{N_1}^{\Gamma_0(M)} \cap T_{N_2}^{\Gamma_0(M)} $, elliptic curves $ E_{\tau}, E_{\tau^\prime} $ have complex multiplication and $ d(E_{\tau}^{\Gamma_0(M)}), d(E_{\tau^\prime}^{\Gamma_0(M)}) \ge -4N_1 N_2$.
\end{proposition}

\begin{proof}
	If $T_{N_1}^{\Gamma_0(M)} $ and $T_{N_2}^{\Gamma_0(M)} $ do not intersect properly, they contain a common component $ V(\Psi_{g}^{\Gamma_0(M)}) $ for some $ g= N_1/(N_1^{\prime})^2 = N_2/(N_2^\prime)^2 $ such that $ (N_1^\prime, M) = (N_2^\prime, M)=1 $.
	Then $N_1 N_2 = (g N_1^\prime N_2^\prime)^2$ is a square.
	
	Conversely, suppose that $N_1 N_2 $ is a square. 
	Let $ g \coloneqq (N_1, N_2), N_1= g(N_1^{\prime})^2, N_2= g(N_2^\prime)^2$.
	Then $ V(\Psi_{g}^{\Gamma_0(M)}) $ is a common component of $T_{N_1}^{\Gamma_0(M)} $ and $T_{N_2}^{\Gamma_0(M)} $.
	
	Suppose that $T_{N_1}^{\Gamma_0(M)} $ and $T_{N_2}^{\Gamma_0(M)} $ intersect properly and $ (t(\tau), t(\tau^\prime) \in T_{N_1}^{\Gamma_0(M)} \cap T_{N_2}^{\Gamma_0(M)} $.
	Then, there exist isogenies $ f_1, f_2 \colon E_{\tau} \to E_{\tau^\prime} $ such that $ f_i\left(\left\langle \frac{1}{M} + \Lambda_{\tau} \right\rangle\right)= \left\langle \frac{1}{M} + \Lambda_{\tau^\prime} \right\rangle, \deg f_i = N_i $ for each $ i=1, 2 $.
	The degree of an isogeny $ \hat{f}_2 \circ f_1 \in E_{\tau}^{\Gamma_0(M)}$ is $N_1 N_2 $, which is not a square.
	Thus  $ \hat{f}_2 \circ f_1 \notin \Z$ and $E_{\tau} $ have complex multiplication.
	Let $\alpha \in \Z[\mathrm{lcm}(a, M)\tau] $ correspond $ \hat{f}_2 \circ f_1 $.
	Since  $d(E_{\tau}^{\Gamma_0(M)}) \mid \trace(\alpha)^2 -4\norm(\alpha) <0$ by Lemma \ref{prop:isog_disc},
	\[
	0 > d(E_{\tau}^{\Gamma_0(M)}) \ge \trace(\alpha)^2 -4\norm(\alpha) \ge  -4\norm(\alpha) = -4N_1 N_2.
	\]
	
	Similarly, it follows that $d(E_{\tau^\prime}^{\Gamma_0(M)}) \ge -4N_1 N_2$ by considering $ \hat{f}_1 \circ f_2 $.
\end{proof}

For $ \tau \in \H $, let $ e_\tau \coloneqq (\Gamma_0(M)_\tau : \{\pm 1\}) $ where $ \Gamma_0(M)_\tau $ is the stabilizer of $ \tau $ for $ \Gamma_0(M) $.
This number is the ramification degree of $ t \colon \H \to \C $ at $ \tau $.

\begin{lemma}\label{lem:int_multp}
	Let $N_1, N_2$ be positive integers prime to $M$ and $T_{N_1}^{\Gamma_0(M)} $ and $T_{N_2}^{\Gamma_0(M)} $ intersect properly.
	Then the intersection multiplicity at $ (t(\tau_0), t(\tau_0^\prime)) \in T_{N_1}^{\Gamma_0(M)} \cap T_{N_2}^{\Gamma_0(M)} $ is given by
	\[
	( T_{N_1}^{\Gamma_0(M)} \cdot T_{N_2}^{\Gamma_0(M)})_{(t(\tau_0), t(\tau_0^\prime))}
	= \dfrac{1}{4e_{\tau_0}e_{\tau_0^\prime}} \# \{ (f_1, f_2) \in 
	\Hom^{\Gamma_0(M)}(E_{\tau_0}^{M}, E_{\tau_0^\prime}^{M}) \mid \deg f_i = N_i \}.
	\] 
\end{lemma}

\begin{proof}
	Since $ t(\tau) - t(\tau_0^\prime) $ has zero at $ \tau_0^\prime $ of order $ e_{\tau_0^\prime} $,
	\[
	\begin{split}
	( T_{N_1}^{\Gamma_0(M)} \cdot T_{N_2}^{\Gamma_0(M)})_{(t(\tau_0), t(\tau_0^\prime))}
	&= \dim_\C \C[[X-t(\tau_0), Y-t(\tau_0^\prime)]]/(\Phi_{N_1}^{\Gamma_0(M)}, \Phi_{N_2}^{\Gamma_0(M)}) \\
	&= \frac{1}{e_{\tau_0^\prime}} \dim_\C \C[[X-t(\tau_0), \tau - \tau_0^\prime]]/(\Phi_{N_1}^{\Gamma_0(M)}, \Phi_{N_2}^{\Gamma_0(M)}).
	\end{split}
	\]
	Since $ \Phi_{N_i}^{\Gamma_0(M)}(X, t(\tau))
	= \prod_{A_i \in I^{\Gamma_0(M)}_{N_i, \mathrm{mat}}
	} (X - t(A_i(\tau))) $,
	\[
	\begin{split}
	&( T_{N_1}^{\Gamma_0(M)} \cdot T_{N_2}^{\Gamma_0(M)})_{(t(\tau_0), t(\tau_0^\prime))} \\
	&= \frac{1}{e_{\tau_0^\prime}} 
	\sum_{\substack{
			A_1 \in I^{\Gamma_0(M)}_{N_1, \mathrm{mat}},\\
			A_2 \in I^{\Gamma_0(M)}_{N_2, \mathrm{mat}}
		} 
	}
	\dim_\C \C[[X-t(\tau_0), \tau - \tau_0^\prime]] /(X - t(A_1(\tau)), X - t(A_2(\tau))) \\
	&= \frac{1}{e_{\tau_0^\prime}}
	\sum_{\substack{
			A_1 \in I^{\Gamma_0(M)}_{N_1, \mathrm{mat}},\\
			A_2 \in I^{\Gamma_0(M)}_{N_2, \mathrm{mat}}, \\
			t(\tau_0)=t(A_1(\tau_0^\prime))=t(A_2(\tau_0^\prime))
		} 
	}
	(\text{the order of $ t(A_1(\tau)) - t(A_2(\tau)) $ at $ \tau_0^\prime $ }).
	\end{split}
	\]
	
	Let $ A_1 \in I^{\Gamma_0(M)}_{N_1, \mathrm{mat}}, A_2 \in I^{\Gamma_0(M)}_{N_2, \mathrm{mat}}, t(\tau_0)=t(A_1(\tau_0^\prime))=t(A_2(\tau_0^\prime)) $.
	Then $\tau_0, A_1(\tau_0^\prime), A_2(\tau_0^\prime) $ are $\Gamma_0(M)$-equivalent.
	By Lemma \ref{lem:mat_rep}, we can assume that
	$ \tau_0 = A_1(\tau_0^\prime) = A_2(\tau_0^\prime)$ and may write $ A_i=
	\begin{pmatrix} a_i & b_i \\ c_i & d_i \end{pmatrix}, c_2=0 $.
	By the same argument in proof of Theorem 2.1 in \cite{Vog}, the order of $ t(A_1(\tau)) - t(A_2(\tau)) $ at $ \tau_0^\prime $ is $ e_{\tau_0} $.
	Thus
	\[
	\begin{split}
	&( T_{N_1}^{\Gamma_0(M)} \cdot T_{N_2}^{\Gamma_0(M)})_{(t(\tau_0), t(\tau_0^\prime))} \\
	&= \frac{e_{\tau_0}}{e_{\tau_0^\prime}}
	\frac{\# \{ (f_1, f_2) \in 
		\Hom^{\Gamma_0(M)}(E_{\tau_0}^{M}, E_{\tau_0^\prime}^{M}) \mid \deg f_i = N_i \}}{(\# \Aut(E_{\tau^\prime}^{\Gamma_0(M)}))^2} \\
	&= \dfrac{1}{4e_{\tau_0}e_{\tau_0^\prime}} \# \{ (f_1, f_2) \in 
	\Hom^{\Gamma_0(M)}(E_{\tau_0}^{M}, E_{\tau_0^\prime}^{M}) \mid \deg f_i = N_i \}.
	\end{split}
	\]
\end{proof}

\subsection{Isogenies between elliptic curves with level structure for $\Gamma_0(M)$}

We prepare several lemmata in order to calculate the intersection number of modular polynomials, which is a sum of intersection multiplicities calculated in Lemma \ref{lem:int_multp}.

In this subsection, we fix positive integers $N_1, N_2$ prime to $M$, an integer $x$ such that $ D \coloneqq 4N_1 N_2 - x^2 >0$ and $ \tau \in \H $ such that $ d^2 \cdot d(E_{\tau}^{\Gamma_0(M)}) =-D $ for some $ d \in \Z_{>0} $.
Let $ a, b, c $ be integers such that $ a\tau^2 +b\tau +c =0, (a, b, c)=1 $.
Then, $ \left(\frac{M}{(a, M)}d\right)^2 (4ac-b^2) =D $ by Lemma \ref{lem:isog_ell_lev}\ref{item:lem:isog_ell_lev1}.
Let $ \beta $ be the unique root of $ X^2 -xX+N_1 N_2 $ on $ \H $ and $ g \in \End(E_{\tau_0}^{M}) $ be the isogeny defined by multiplying by $ \beta $.
We can write
\[
\beta = \frac{1}{2} \left( x+ \frac{M}{(a, M)} bd \right) + \frac{M}{(a, M)} ad\tau.
\]
We define 
\[
\mathcal{D}(\tau, x) \coloneqq \{ (E_{\tau^\prime}^{\Gamma_0(M)}, f_1, f_2) \mid \tau^\prime \in \H, f_i \in \Hom(E_{\tau}^{\Gamma_0(M)}, E_{\tau^\prime}^{\Gamma_0(M)}), \deg f_i = N_i, (f_1, f_2)=x \}/\cong,
\]
\[
\mathcal{D}_i(\tau, x) \coloneqq \{ (E_{\tau^\prime}^{\Gamma_0(M)}, f) \mid \tau^\prime \in \H, f \in \Hom(E_{\tau}^{\Gamma_0(M)}, E_{\tau^\prime}^{\Gamma_0(M)}), \deg f = N_i, N_i \mid f \circ g \}/\cong, \quad i=1, 2
\]
where two pairs $ (E_{\tau^\prime}^{\Gamma_0(M)}, f_1^{\prime}, f_2^{\prime}), (E_{\tau^{\prime\prime}}^{M}, f_1^{\prime\prime}, f_2^{\prime\prime})  $ are called isomorphic if there exists an isomorphism $ \varphi \colon E_{\tau^\prime}  \xrightarrow{\sim}  E_{\tau^{\prime\prime}} $ such that $ f_1^{\prime\prime}=\varphi \circ f_1^{\prime}, f_2^{\prime\prime}=\phi \circ f_2^{\prime} $ and two pairs $ (E_{\tau^\prime}^{\Gamma_0(M)}, f^{\prime}), (E_{\tau^{\prime\prime}}^{M}, f^{\prime\prime})  $ are called isomorphic if there exists an isomorphism $ \varphi \colon E_{\tau^\prime}  \xrightarrow{\sim}  E_{\tau^{\prime\prime}} $ such that $ f^{\prime\prime}=\varphi \circ f^{\prime} $.
Also we define
\[
\mathcal{D}^{\Gamma_0(M)}(\tau, x) \coloneqq \{ [E_{\tau^\prime}^{\Gamma_0(M)}, f_1, f_2] \in \mathcal{D}(\tau, x) \mid f_i \in \Hom^{\Gamma_0(M)}(E_{\tau}, E_{\tau^\prime}) \},
\]
\[
\mathcal{D}_i^{\Gamma_0(M)}(\tau, x) \coloneqq \{ [E_{\tau^\prime}^{\Gamma_0(M)}, f] \in \mathcal{D}_i(\tau, x) \mid f \in \Hom^{\Gamma_0(M)}(E_{\tau}, E_{\tau^\prime}) \}, \quad i=1, 2.
\]

For an isogeny $ f \colon E \to E^\prime $ between elliptic curves and an integer $ N $, denote $ N \mid f $ if there exists an isogeny $ f^{\prime} \colon E \to E^\prime $ such that $ f = f^{\prime} \circ [N]_{E} $.

\begin{lemma}\label{lem:bij_between_D}
	The following map is bijective:
	\[
	\begin{array}{ccc}
	\mathcal{D}(\tau, x)  & \longrightarrow & \mathcal{D}_1(\tau, x) \sqcup \mathcal{D}_2(\tau, x) \\
	{[E_{\tau^\prime}^{\Gamma_0(M)}, f_1, f_2]} & \longmapsto &
	\begin{cases}
	[E_{\tau^\prime}^{\Gamma_0(M)}, f_1] & \text{ if } \hat{f}_1 \circ f_2 =g \\
	[E_{\tau^\prime}^{\Gamma_0(M)}, f_2] & \text{ if } \hat{f}_2 \circ f_1 =g.
	\end{cases}
	\end{array}
	\]
	This map induces a bijection between $ \mathcal{D}^{\Gamma_0(M)}(\tau, x)  $ and $ \mathcal{D}_1^0(\tau, x) \sqcup \mathcal{D}_2^0(\tau, x) $.
\end{lemma}

\begin{proof}
	Let $ \alpha_{1}, \alpha_{2} \in \C $ correspond to $ f_1, f_2 $.
	Then, $  \bar{\alpha}_{1} \alpha_{2} $ and $  \bar{\alpha}_{2} \alpha_{1} $ are roots of $ X^2 -xX+N_1 N_2 $.
	Thus $ \hat{f}_1 \circ f_2 =g $ or $ \hat{f}_2 \circ f_1 $ is equal to $ g $.
\end{proof}

Let $ e \coloneqq (N_1, N_2, x) $.
We can prove the following lemma by Lemma \ref{lem:isog_ell_lev} and the same argument in proof of Lemma 2.1 in \cite{Gor}.

\begin{lemma}\label{lem:D_corresp_Z}
	For $ i=1, 2 $, there exists a bijection between the following sets:
	\begin{enumerate}
		\item $ \mathcal{D}_i(\tau, x) $,
		\item \[
		\left\{ X= \begin{pmatrix} p & q \\ 0 & s \end{pmatrix} \in \mathrm{M}_2(\Z) \relmiddle| X \text{ satisfies \ref{item:lem:D_corresp_Z1}, \ref{item:lem:D_corresp_Z2} } \right\}
		\] where
		\begin{enumerate}
			\item \label{item:lem:D_corresp_Z1} $ ps=N_i, Z \coloneqq \frac{1}{p}\left( \frac{M}{(a, M)}ad, \frac{1}{2}\left( x- \frac{M}{(a, M)}bd \right), N_i \right) \mid (e, d)$
			\item \label{item:lem:D_corresp_Z2} $0 \le q < s$ and $q \equiv q^\prime \bmod \frac{s}{Z}$ for some $ q^\prime \in \Z/\frac{s}{Z}\Z$ depending on $ Z $.
		\end{enumerate}
	\end{enumerate}
	
	Thus the number of elements of these sets is given by
	\[
	\sigma \left( \left( e, \frac{M}{(a, M)}d \right) \right)
	= \sum_{ Z \mid \left( e, \frac{M}{(a, M)}d \right) } Z.
	\] 
	
	If $ [E_{\tau^\prime}^{\Gamma_0(M)}, f], A = \begin{pmatrix}  p & q \\ 0 & s \end{pmatrix}  $ and $ Z $ correspond under this bijection, then 
	\[
	p = \dfrac{1}{Z}\left( \frac{M}{(a, M)}ad, \frac{1}{2}\left( x- \frac{M}{(a, M)}bd \right), N_i \right)
	\]
	and $ [E_{\tau^\prime}^{\Gamma_0(M)}, f] $ is isomorphic to a pair of $E_{A(\tau)}^{M} $ and an isogeny from $E_{\tau}^{\Gamma_0(M)} $ to $E_{A(\tau)}^{M} $ defined by multiplying by $ p $.
	
	Thus  
	\[
	\# \mathcal{D}_i^{\Gamma_0(M)}(\tau, x)
	= \sum_{ Z \mid \left( e, \frac{M}{(a, M)}d \right), \left(  \frac{L_i}{Z}, M \right)=1 } Z
	\]
	where 
	\[
	L_i \coloneqq  \left( \frac{M}{(a, M)}ad, \frac{1}{2}\left( x- \frac{M}{(a, M)}bd \right), N_i \right).
	\]
\end{lemma}

\subsection{Proof of Theorem \ref{thm:main} and Theorem \ref{thm:main}}

We are now ready to prove Theorem \ref{thm:main} by using results obtained in previous sections.

\begin{proof}[Proof of Theorem \ref{thm:main}.]
	Suppose that $ N_1, N_2 $ are integers prime to $ M $.
	The symbol $ [\tau] $ denotes the element $ \Gamma_0(M)\tau $ in $ Y_0(M) $.
	By Lemma \ref{lem:int_multp}, the intersection number of curves defined by $\Phi_{N_1}^{\Gamma_0(M)}$ and $\Phi_{N_2}^{\Gamma_0(M)}$ on the image of $Y_0(M) \times Y_0(M)$ in $\C \times \C$ under $t \times t$ is
	\[
	\begin{split}
	&(\Phi_{N_1}^{\Gamma_0(M)}\cdot\Phi_{N_2}^{\Gamma_0(M)}) \\
	&= \sum_{[\tau_0],[\tau_0^\prime] \in Y_0(M)}
	( T_{N_1}^{\Gamma_0(M)} \cdot T_{N_2}^{\Gamma_0(M)})_{(t(\tau_0), t(\tau_0^\prime))} \\
	&= \sum_{[\tau_0],[\tau_0^\prime]}
	\frac{\# \{ (f_1, f_2) \in \Hom^{\Gamma_0(M)}(E_{\tau_0}^{M}, E_{\tau_0^\prime}^{M}) \mid \deg f_i = N_i \}}
	{4e_{\tau_0}e_{\tau_0^\prime}} \\
	&= \sum_{\substack{
			x \in \Z, \\
			x^2 < 4N_1 N_2
	}}
	\sum_{[\tau_0],[\tau_0^\prime] }
	\frac{\# \left\{ (f_1, f_2) \in \Hom^{\Gamma_0(M)}(E_{\tau_0}^{M}, E_{\tau_0^\prime}^{M}) \relmiddle|
		\begin{array}{l}
		\deg (n_1 f_1 + n_2 f_2) \\
		= N_1 n_1^2 +x n_1 n_2 + N_2 n_2^2 \\
		\text{ for all } n_1, n_2 \in \Z  
		\end{array} 
		\right\}}
	{4e_{\tau_0}e_{\tau_0^\prime}} \\
	&= \sum_{\substack{
			x \in \Z, \\
			x^2 < 4N_1 N_2
	}}
	\sum_{d^2 \mid 4N_1 N_2 - x^2}
	\sum_{\substack{
			[\tau_0]  \in Y_0(M) \\
			\tau_0 \text{ imaginary quadratic } \\
			d(E_{\tau_0}^M) = - \frac{4N_1 N_2 - x^2}{d^2}
	}}
	\frac{\# \mathcal{D}^{\Gamma_0(M)}(\tau_0, x) }{2e_{\tau_0}}.
	\end{split}
	\]
	
	If $E_{\tau_0}^{M} $ is an elliptic curve with complex multiplication with level structure for $\Gamma_0(M)$ and $ a, b, c $ be integers such that $ a\tau_0^2 +b\tau_0 +c =0, (a, b, c)=1 $, then an isomorphic class of $E_{\tau_0}^{M} $ corresponds to an $\Gamma_0(M)$-equivalence class of a primitive positive definite binary quadratic form $ [a, b, c] $ over $ \Z $.
	By Lemma \ref{lem:bij_between_D}, this is
	\[
	\sum_{x \in \Z,\ x^2 < 4N_1 N_2}
	\sum_{d^2 \mid 4N_1 N_2 - x^2}
	\sum_{\substack{
			[\tau_0] \in Y_0(M) \\
			\tau_0 \text{ imaginary quadratic } \\
			d(E_{\tau_0}^M) = - \frac{4N_1 N_2 - x^2}{d^2}
	}}
	\frac{\# \mathcal{D}_1^0(\tau_0, x) + \# \mathcal{D}_2^0(\tau_0, x)}{2e_{\tau_0}}.
	\]
	By Lemma \ref{lem:D_corresp_Z}, this is
	\[
	\sum_{x \in \Z,\ x^2 < 4N_1 N_2}
	\sum_{d^2 \mid 4N_1 N_2 - x^2}
	\sum_{\substack{
			[\tau_0] \in Y_0(M) \\
			\tau_0 \text{ imaginary quadratic } \\
			d(E_{\tau_0}^M) = - \frac{4N_1 N_2 - x^2}{d^2}
	}}
	\sum_{i=1, 2}
	\sum_{\substack{
			Z \mid \left( N_1, N_2, x, \frac{M}{(a, M)}d \right) \\
			(\frac{L_i}{Z}, M)=1
	}}
	\frac{Z}{2e_{\tau_0}}
	\]
	where $ a, b, c $ be integers such that $ a\tau_0^2 +b\tau_0 +c =0, (a, b, c)=1 $ and 
	\[
	L_i \coloneqq  \left( \frac{M}{(a, M)}ad, \frac{1}{2}\left( x- \frac{M}{(a, M)}bd \right), N_i \right).
	\]
	Since $ (N_1, M)=(N_2, M)=1 $, we can eliminate the condition $ (\frac{L_i}{Z}, M)=1 $ in the sum.
	By Lemma \ref{lem:e_D_corresp}, this sum is
	\[
	\sum_{x^2 < 4N_1 N_2} \sum_{Z \mid (N_1, N_2, x)}
	Z \cdot H^M \left( \frac{4N_1 N_2- x^2}{Z^2} \right).
	\]
	For the case when $ M=p $ is a prime, by replacing $ Z $ by $ d $ and applying Proposition \ref{prop:sum_of_H^M_and_H}, this sum is
	\[
	\sum_{x^2 < 4N_1 N_2} A^p(4N_1 N_2- x^2) \sum_{d \mid (N_1, N_2, x)}
	d \cdot H \left( \frac{4N_1 N_2- x^2}{d^2} \right).
	\]
	The proof is complete.
\end{proof}

\begin{proof}[Proof of Theorem \ref{thm:main2}.]
	As mentioned in Section 1, let us follow the notation in \cite{Nagaoka}. By Theorem 2.1.1 of \cite{Nagaoka} we see that 
	$$C(T)=288\sum_{d \mid T}dH\left( \frac{\det(2T)}{d^2} \right)$$
	for any $T\in {\rm Sym}_2(\Z)_{>0}$. 
	Notice that $\chi_{T}(p)=\chi_{d^{-1}T}(p)$ for any $d \mid T$ and $p \mid M$ since the imaginary quadratic field corresponding to $ T $ is equal to one corresponding to $ d^{-1}T $.
	Then the claim follows from this with Theorem \ref{thm:main}.
\end{proof}

\section*{Acknowledgments}
	I would like to show my greatest appreciation to Professor Takuya Yamauchi, for introducing me to this topic and giving many advices. I am deeply grateful to Dr. Takumi Nakayama and Dr. Hiroki Sakurai for many discussion and sharing much time for studying modular curves and modular forms.

\bibliographystyle{alpha}
\bibliography{myrefs_for_work}

\end{document}